%% file: output.tex
\documentclass[11pt]{article}

\usepackage{amsmath}
\makeatletter
\let\save@mathaccent\mathaccent
\newcommand*\if@single[3]{%
  \setbox0\hbox{${\mathaccent"0362{#1}}^H$}%
  \setbox2\hbox{${\mathaccent"0362{\kern0pt#1}}^H$}%
  \ifdim\ht0=\ht2 #3\else #2\fi
  }
\newcommand*\rel@kern[1]{\kern#1\dimexpr\macc@kerna}
\newcommand*\widebar[1]{\@ifnextchar^{{\wide@bar{#1}{0}}}{\wide@bar{#1}{1}}}
\newcommand*\wide@bar[2]{\if@single{#1}{\wide@bar@{#1}{#2}{1}}{\wide@bar@{#1}{#2}{2}}}
\newcommand*\wide@bar@[3]{%
  \begingroup
  \def\mathaccent##1##2{%
    \let\mathaccent\save@mathaccent
    \if#32 \let\macc@nucleus\first@char \fi
    \setbox\z@\hbox{$\macc@style{\macc@nucleus}_{}$}%
    \setbox\tw@\hbox{$\macc@style{\macc@nucleus}{}_{}$}%
    \dimen@\wd\tw@
    \advance\dimen@-\wd\z@
    \divide\dimen@ 3
    \@tempdima\wd\tw@
    \advance\@tempdima-\scriptspace
    \divide\@tempdima 10
    \advance\dimen@-\@tempdima
    \ifdim\dimen@>\z@ \dimen@0pt\fi
    \rel@kern{0.6}\kern-\dimen@
    \if#31
      \overline{\rel@kern{-0.6}\kern\dimen@\macc@nucleus\rel@kern{0.4}\kern\dimen@}%
      \advance\dimen@0.4\dimexpr\macc@kerna
      \let\final@kern#2%
      \ifdim\dimen@<\z@ \let\final@kern1\fi
      \if\final@kern1 \kern-\dimen@\fi
    \else
      \overline{\rel@kern{-0.6}\kern\dimen@#1}%
    \fi
  }%
  \macc@depth\@ne
  \let\math@bgroup\@empty \let\math@egroup\macc@set@skewchar
  \mathsurround\z@ \frozen@everymath{\mathgroup\macc@group\relax}%
  \macc@set@skewchar\relax
  \let\mathaccentV\macc@nested@a
  \if#31
    \macc@nested@a\relax111{#1}%
  \else
    \def\gobble@till@marker##1\endmarker{}%
    \futurelet\first@char\gobble@till@marker#1\endmarker
    \ifcat\noexpand\first@char A\else
      \def\first@char{}%
    \fi
    \macc@nested@a\relax111{\first@char}%
  \fi
  \endgroup
}
\makeatother
\usepackage[authoryear]{natbib}
\usepackage{yhmath}
\usepackage{amssymb}
\usepackage{amsfonts}
\usepackage{mathrsfs}
\usepackage{graphicx}
\usepackage{color,xcolor}
\usepackage{float}
\usepackage{subcaption}
\usepackage[normalem]{ulem}
\usepackage[linktocpage,colorlinks,linkcolor=blue,anchorcolor=blue, citecolor=blue,urlcolor=blue]{hyperref}
\usepackage{pdfpages}
\usepackage[inline]{enumitem}
\usepackage{multirow}
\usepackage{makecell}
\usepackage{breakcites}
\usepackage{bbm}
\usepackage{wrapfig}
\usepackage{wrapstuff}
\usepackage[thinc]{esdiff}
\allowdisplaybreaks[4]
\usepackage{rotating}
\usepackage{scalerel,stackengine}
\usepackage{cases}
\usepackage{mleftright}
\usepackage[inline]{asymptote}
\usepackage{lmodern}
\usepackage{microtype}
\newtheorem{theorem}{Theorem}[section]
\newtheorem{definition}[theorem]{Definition}
\newtheorem{lemma}[theorem]{Lemma}

\newtheorem{corollary}[theorem]{Corollary}
\newtheorem{conjecture}[theorem]{Conjecture}
\newtheorem{proposition}[theorem]{Proposition}

\def\bd{\boldsymbol}
\def\pn{\par\smallskip\noindent}

\newenvironment{myproof}[1] {\pn {\em Proof of {#1}.}}{\hfill $\Box$ \vskip 0.2truein}

\DeclareMathAlphabet\mathbfcal{OMS}{cmsy}{b}{n}

\newcommand{\R}{\mathbb{R}}

\newcommand{\bbJ}{\mathbb{J}}
\newcommand{\bbP}{\mathbb{P}}

\newcommand{\bbB}{\mathbb{B}}
\newcommand{\bbF}{\mathbb{F}}

\newcommand{\bbD}{\mathbb{D}}

\newcommand{\bbO}{\mathbb{O}}

\newcommand{\bbS}{\mathbb{S}}

\newcommand{\bbH}{\mathbb{H}}

\newcommand{\bbX}{\mathbb{X}}

\newcommand{\bA}{\boldsymbol{A}}
\newcommand{\bbA}{\mathbb{A}}

\newcommand{\bdd}{\boldsymbol{d}}

\newcommand{\bI}{\boldsymbol{I}}

\newcommand{\be}{\boldsymbol{e}}

\newcommand{\bH}{\boldsymbol{H}}

\newcommand{\bG}{\boldsymbol{G}}
\newcommand{\bx}{\boldsymbol{x}}
\newcommand{\by}{\boldsymbol{y}}
\newcommand{\bz}{\boldsymbol{z}}

\newcommand{\bv}{\boldsymbol{v}}
\newcommand{\bw}{\boldsymbol{w}}

\newcommand{\bU}{\boldsymbol{U}}

\newcommand{\bX}{\boldsymbol{X}}

\newcommand{\ddet}{\textnormal{det}}

\newcommand{\ddetzr}{\textnormal{det-zr}}

\newcommand{\dC}{\textnormal{C}}

\newcommand{\T}{\textnormal{T}}

\usepackage{booktabs}
\usepackage{graphicx}
\usepackage{stmaryrd}
\usepackage{mathtools}
\def\multiset#1#2{\ensuremath{\left(\kern-.3em\left(\genfrac{}{}{0pt}{}{#1}{#2}\right)\kern-.3em\right)}}

\usepackage{multicol}
\usepackage{nicematrix}
\usepackage{stackengine}
\usepackage{rotating}

\stackMath

\usepackage{algorithm}
\usepackage[noend]{algorithmic}

\usepackage{stackengine}
\stackMath
\newcommand\tsup[2][2]{%
 \def\useanchorwidth{T}%
  \ifnum#1>1%
    \stackon[-.5pt]{\tsup[\numexpr#1-1\relax]{#2}}{\scriptscriptstyle\sim}%
  \else%
    \stackon[.5pt]{#2}{\scriptscriptstyle\sim}%
  \fi%
}

\newcommand\restr[2]{{
  \left.\kern-\nulldelimiterspace 
  #1 
  \littletaller 
  \right|_{#2} 
  }}

\newcommand{\littletaller}{\mathchoice{\vphantom{\big|}}{}{}{}}

\usepackage{tikzit}
\input{sample.tikzstyles}

\usepackage{tikz}
\usepackage{tikz-3dplot}
\usetikzlibrary{calc}
\usepackage{pgfplots}
\usepackage{xxcolor}
\pgfplotsset{compat=1.16}
\usetikzlibrary{external}
\usepgfplotslibrary{fillbetween}
\usetikzlibrary{patterns}
\pgfdeclareradialshading[tikz@ball]{ball}{\pgfqpoint{0bp}{0bp}}{%
 color(0bp)=(tikz@ball!0!white);
 color(7bp)=(tikz@ball!0!white);
 color(15bp)=(tikz@ball!70!black);
 color(20bp)=(black!70);
 color(30bp)=(black!70)}
\makeatother

\tikzset{viewport/.style 2 args={
    x={({cos(-#1)*1cm},{sin(-#1)*sin(#2)*1cm})},
    y={({-sin(-#1)*1cm},{cos(-#1)*sin(#2)*1cm})},
    z={(0,{cos(#2)*1cm})}
}}

\pgfplotsset{only foreground/.style={
    restrict expr to domain={rawx*\CameraX + rawy*\CameraY + rawz*\CameraZ}{-0.05:100},
}}
\pgfplotsset{only background/.style={
    restrict expr to domain={rawx*\CameraX + rawy*\CameraY + rawz*\CameraZ}{-100:0.05}
}}

\def\addFGBGplot[#1]#2;{
    \addplot3[#1,only background, opacity=0.25] #2;
    \addplot3[#1,only foreground] #2;
}

\usetikzlibrary{3d}
\usetikzlibrary{calc}
\usetikzlibrary{babel} 

\pgfmathsetmacro\xx{1/sqrt(2)}
\pgfmathsetmacro\xy{1/sqrt(6)}
\pgfmathsetmacro\zy{sqrt(2/3)}

\usetikzlibrary{3d}
\usetikzlibrary{calc, shadings} 
\usetikzlibrary{positioning,arrows.meta}
\usetikzlibrary{trees}
\usepgfplotslibrary{fillbetween}
\DeclareMathAlphabet\mathbfcal{OMS}{cmsy}{b}{n}

\makeatletter
\def\tikz@lib@cuboid@get#1{\pgfkeysvalueof{/tikz/cuboid/#1}}

\def\tikz@lib@cuboid@setup{%
   \pgfmathsetlengthmacro{\vxx}%
      {\tikz@lib@cuboid@get{xscale}*cos(\tikz@lib@cuboid@get{xangle})*1cm}
   \pgfmathsetlengthmacro{\vxy}%
      {\tikz@lib@cuboid@get{xscale}*sin(\tikz@lib@cuboid@get{xangle})*1cm}
   \pgfmathsetlengthmacro{\vyx}%
      {\tikz@lib@cuboid@get{yscale}*cos(\tikz@lib@cuboid@get{yangle})*1cm}
   \pgfmathsetlengthmacro{\vyy}%
      {\tikz@lib@cuboid@get{yscale}*sin(\tikz@lib@cuboid@get{yangle})*1cm}
   \pgfmathsetlengthmacro{\vzx}%
      {\tikz@lib@cuboid@get{zscale}*cos(\tikz@lib@cuboid@get{zangle})*1cm}
   \pgfmathsetlengthmacro{\vzy}%
      {\tikz@lib@cuboid@get{zscale}*sin(\tikz@lib@cuboid@get{zangle})*1cm}
}

\def\tikz@lib@cuboid@draw#1--#2--#3\pgf@stop{%
    \begin{scope}[join=bevel,x={(\vxx,\vxy)},y={(\vyx,\vyy)},z={(\vzx,\vzy)}]
       \begin{scope}[canvas is yz plane at x=#1]
          \draw[cuboid/all faces,cuboid/edges,cuboid/right face] 
                (0,0) -- ++(#2,0) -- ++(0,-#3) -- ++(-#2,0) -- cycle;
          \draw[cuboid/all grids,cuboid/right grid] (0,0) grid (#2,-#3);
       \end{scope}
       \begin{scope}[canvas is xy plane at z=0]
          \draw[cuboid/all faces,cuboid/edges,cuboid/front face] 
                (0,0) -- ++(#1,0) --  ++(0,#2) -- ++(-#1,0) -- cycle;
          \draw[cuboid/all grids,cuboid/front grid] (0,0) grid (#1,#2);
       \end{scope}
       \begin{scope}[canvas is xz plane at y=#2]
          \draw[cuboid/all faces,cuboid/edges,cuboid/top face] 
                (0,0) -- ++(#1,0) --  ++(0,-#3) -- ++(-#1,0) -- cycle;
          \draw[cuboid/all grids,cuboid/top grid] (0,0) grid (#1,-#3);
       \end{scope}
       \draw[cuboid/hidden edges] (0,#2,-#3) -- (0,0,-#3) -- (0,0,0) 
                (0,0,-#3) -- ++(#1,0,0);
       \begin{scope}[canvas is yz plane at x=#1]
          \draw[cuboid/all faces,cuboid/right face,cuboid/edges,fill opacity=0] 
                (0,0) -- ++(#2,0) -- ++(0,-#3) -- ++(-#2,0) -- cycle;
       \end{scope}
       \begin{scope}[canvas is xy plane at z=0]
          \draw[cuboid/all faces,cuboid/front face,cuboid/edges,fill opacity=0] 
                (0,0) -- ++(#1,0) --  ++(0,#2) -- ++(-#1,0) -- cycle;
       \end{scope}
       \begin{scope}[canvas is xz plane at y=#2]
          \draw[cuboid/all faces,cuboid/top face,cuboid/edges,fill opacity=0] 
                (0,0) -- ++(#1,0) --  ++(0,-#3) -- ++(-#1,0) -- cycle;
       \end{scope}
       \path (0,#2,0) coordinate (-left top front)
                      coordinate (-left front top)
                      coordinate (-top left front)
                      coordinate (-top front left)
                      coordinate (-front top left)
                      coordinate (-front left top);
       \path (0,#2,-#3) coordinate (-left top rear)
                        coordinate (-left rear top)
                        coordinate (-top left rear)
                        coordinate (-top rear left)
                        coordinate (-rear top left)
                        coordinate (-rear left top);
       \path (0,0,-#3) coordinate (-left bottom rear)
                       coordinate (-left rear bottom)
                       coordinate (-bottom left rear)
                       coordinate (-bottom rear left)
                       coordinate (-rear bottom left)
                       coordinate (-rear left bottom);
       \path (0,0,0) coordinate (-left bottom front)
                     coordinate (-left front bottom)
                     coordinate (-bottom left front)
                     coordinate (-bottom front left)
                     coordinate (-front bottom left)
                     coordinate (-front left bottom);
       \path (#1,#2,0) coordinate (-right top front)
                       coordinate (-right front top)
                       coordinate (-top right front)
                       coordinate (-top front right)
                       coordinate (-front top right)
                       coordinate (-front right top);
       \path (#1,#2,-#3) coordinate (-right top rear)
                         coordinate (-right rear top)
                         coordinate (-top right rear)
                         coordinate (-top rear right)
                         coordinate (-rear top right)
                         coordinate (-rear right top);
       \path (#1,0,-#3) coordinate (-right bottom rear)
                        coordinate (-right rear bottom)
                        coordinate (-bottom right rear)
                        coordinate (-bottom rear right)
                        coordinate (-rear bottom right)
                        coordinate (-rear right bottom);
       \path (#1,0,0) coordinate (-right bottom front)
                      coordinate (-right front bottom)
                      coordinate (-bottom right front)
                      coordinate (-bottom front right)
                      coordinate (-front bottom right)
                      coordinate (-front right bottom);
       \coordinate (-left center) at (0,.5*#2,-.5*#3);
       \coordinate (-right center) at (#1,.5*#2,-.5*#3);
       \coordinate (-top center) at (.5*#1,#2,-.5*#3);
       \coordinate (-bottom center) at (.5*#1,0,-.5*#3);
       \coordinate (-front center) at (.5*#1,.5*#2,0);
       \coordinate (-rear center) at (.5*#1,.5*#2,-#3);
       \coordinate (-center) at (.5*#1,.5*#2,-.5*#3);
       \path (0,#2,-.5*#3) coordinate (-left top center) 
                           coordinate (-top left center);
       \path (.5*#1,#2,-#3) coordinate (-top rear center)
                            coordinate (-rear top center);
       \path (#1,#2,-.5*#3) coordinate (-right top center)
                            coordinate (-top right center);
       \path (.5*#1,#2,0) coordinate (-top front center)
                          coordinate (-front top center);
       \path (0,0,-.5*#3) coordinate (-left bottom center) 
                           coordinate (-bottom left center);
       \path (.5*#1,0,-#3) coordinate (-bottom rear center)
                            coordinate (-rear bottom center);
       \path (#1,0,-.5*#3) coordinate (-right bottom center)
                            coordinate (-bottom right center);
       \path (.5*#1,0,0) coordinate (-bottom front center)
                          coordinate (-front bottom center);
       \path (0,.5*#2,0) coordinate (-left front center) 
                           coordinate (-front left center);
       \path (0,.5*#2,-#3) coordinate (-left rear center)
                            coordinate (-rear left center);
       \path (#1,.5*#2,0) coordinate (-right front center)
                            coordinate (-front right center);
       \path (#1,.5*#2,-#3) coordinate (-right rear center)
                          coordinate (-rear right center);
    \end{scope}
}

\tikzset{
  pics/cuboid/.style = {
    setup code = \tikz@lib@cuboid@setup,
    background code = \tikz@lib@cuboid@draw#1\pgf@stop
  },
  pics/cuboid/.default={1--1--1},
  cuboid/.is family,
  cuboid,
  all faces/.style={fill=white},
  all grids/.style={draw=none},
  front face/.style={},
  front grid/.style={},
  right face/.style={},
  right grid/.style={},
  top face/.style={},
  top grid/.style={},
  edges/.style={},
  hidden edges/.style={draw=none},
  xangle/.initial=0,
  yangle/.initial=90,
  zangle/.initial=210,
  xscale/.initial=1,
  yscale/.initial=1,
  zscale/.initial=0.5
}

\newcommand{\tikzcuboidreset}{
\tikzset{cuboid,
  all faces/.style={fill=white},
  all grids/.style={draw=none},
  front face/.style={},
  front grid/.style={},
  right face/.style={},
  right grid/.style={},
  top face/.style={},
  top grid/.style={},
  edges/.style={},
  hidden edges/.style={draw=none},
  xangle=0,
  yangle=90,
  zangle=210,
  xscale=1,
  yscale=1,
  zscale=0.5
}
}

\newcommand{\tikzcuboidset}{\@ifstar\tikzcuboidset@star\tikzcuboidset@nostar} 
\newcommand{\tikzcuboidset@nostar}[1]{\tikzcuboidreset\tikzset{cuboid,#1}}
\newcommand{\tikzcuboidset@star}[1]{\tikzset{cuboid,#1}}
\makeatother

\usetikzlibrary{angles, quotes}

\makeatletter
\newif\ifnobrackets
\renewcommand\@cite[2]{\ifnobrackets\else[\fi{#1\if@tempswa , #2\fi}\ifnobrackets\else]\fi\nobracketsfalse}

\makeatother

\usepackage{lmodern}
\usetikzlibrary{decorations.pathmorphing}
\tikzset{snake it/.style={decorate, decoration=snake}}

\begin{document}

\title{
On the hardness of deterministic second-order optimization
\\
of functions with Lipschitz gradients 
}

\author{
Jiewen GUAN
\thanks{Department of Systems Engineering and Engineering Management, The Chinese University of Hong Kong, Shatin, New Territories, Hong Kong. Email: seemjwguan@gmail.com}
    \and
Anthony Man-Cho SO
\thanks{Department of Systems Engineering and Engineering Management, The Chinese University of Hong Kong, Shatin, New Territories, Hong Kong. Email: manchoso@se.cuhk.edu.hk}
}

\date{\today}

\maketitle

\begin{abstract}

We show that no deterministic zero-respecting algorithm (resp., (general) deterministic algorithm) can compute Goldstein approximate second-order stationary points of functions with Lipschitz continuous gradients within a finite number of (resp., no more than $n-3$ with $n$ being the input dimension) second-order oracle calls. This, among other consequences, shows that deterministic second-order weakly convex optimization is intractable.

\vspace{0.25cm}

\noindent {\bf Keywords:}
Variational analysis, second-order theory, lower bounds, oracle complexity, nonconvex nonsmooth optimization, $C^{1,1}$-functions, weak convexity, 
second-order stationarity


\end{abstract}

\section{Introduction}
Differentiable functions with 
Lipschitz continuous gradients, commonly referred to as $C^{1,1}$-functions, are prevalent across multiple disciplines; e.g., machine learning and statistics. 
As a result, understanding the optimization properties of $C^{1,1}$-functions is of considerable interest.
Owing to the substantial contributions of researchers over recent decades, the first-order analytical and algorithmic properties of 
$C^{1,1}$-optimization have been thoroughly understood; 
see, e.g.,~\citep[Section~2]{nesterov2018lectures},~\citep{carmon2020lower}, and the references therein. 
For instance, it is now folkloric that an $\varepsilon$-approximate first-order stationary point of a (lower-bounded) $C^{1,1}$-function can be found within at most $\mathcal{O}(\varepsilon^{-2})$ steps of gradient descent; see, e.g.,~\citep[Section~1.2.3]{nesterov2018lectures}.

That said, as the gradients of $C^{1,1}$-functions can potentially be nonsmooth (e.g., the Huber loss~\citep{huber1964robust}, squared hinge loss~\citep{hsieh2008dual,chang2008coordinate}, and exponential linear unit~\citep{clevert2016fast}), the 
behavior of second-order methods for optimizing $C^{1,1}$-functions
can be rather
subtle. While substantial research progress has been made along this direction (see, e.g.,~\citep[Section~13]{rockafellar2009variational}, the recent book by~\citet{mordukhovich2024second}, and the references therein), there are still several aspects that 
have received
only
limited attention; e.g., the (oracle) complexity of 
finding approximate second-order stationary points of $C^{1,1}$-functions.
Conceivably, 
as the task of finding 
approximate first-order stationary points of 
Lipschitz functions 
is
challenging (see, e.g.,~\citep{daniilidis2020pathological,zhang2020complexity,kornowski2021nips,tian2021hardness,kornowski2022jmlr,tian2022finite,jordan2023deterministic,tian2024no} for a series of recent advances shedding light on this point), so should the task of finding 
approximate second-order
stationary points of 
$C^{1,1}$-functions. 

In this paper, 
our goal is
to rigorously justify this intuition. 
Towards that end, we first 
consider
the Goldstein second-order subdifferential (Definition~\ref{def:Goldstein-2nd-subdiff}) and its associated stationarity concept (Definition~\ref{def:Goldstein-SOSP}), 
which arise naturally as second-order analogues of the 
first-order
Goldstein
framework~\citep{goldstein1977optimization},
motivated by the goal of identifying an approximate stationarity concept for which complexity guarantees can be established,
as in the first-order setting~\citep{zhang2020complexity}.
Then, we show that
for any number of iterations $m\ge 1$, any dimension $n\ge 2$ (resp., $n\ge m+3$), and any second-order deterministic zero-respecting algorithm (resp., second-order (general) deterministic algorithm) $\mathbfcal{A}$ that seeks solutions over $\R^n$, there exists some $C^{1,1}$-function $f:\R^n\rightarrow\R$ (with fixed Lipschitz constant and initial optimality gap) such that the 
iterates $\bx_0,\ldots,\bx_m$ generated by $\mathbfcal{A}$ when applied to $f$
are not 
Goldstein approximate second-order stationary
up to some fixed tolerances (Theorem~\ref{thm:lb-1} and Corollary~\ref{cor:lb-1}).
Concretely, the class of second-order deterministic zero-respecting algorithms captures most of the standard nonsmooth Newton methods built on different second-order subdifferential constructions~\citep{kummer1988newton,pang1990newton,qi1993convergence,qi1993nonsmooth,hoheisel2012generalized,mordukhovich2021generalized,khanh2023generalized}, whereas the class of second-order (general) deterministic algorithms covers virtually all common deterministic second-order methods,
including the nonconvex cutting-plane method~\citep{fuduli2004minimizing} and grid search~\citep[Section~1.1.3]{nesterov2018lectures}.
These hardness results well align with the fact that, 
without 
additional
assumptions,
no complexity upper bound is known 
for
the deterministic computation of approximate second-order stationary points of $C^{1,1}$-functions.
In particular, a key takeaway 
is that 
finding approximate
second-order stationary points of weakly convex
functions in a deterministic manner
is far from being tractable, as $C^{1,1}$-functions are automatically weakly convex; see, e.g.,~\citep[Proposition~4.12]{vial1983strong}.\footnote{It is noteworthy that weakly convex functions arise widely in machine learning and beyond; e.g., sparse dictionary learning~\citep[Example~2.4]{davis2019stochastic} and robust low-rank matrix recovery~\citep{li2020nonconvex}. For further examples, we refer interested readers to~\citep[Section~2.1]{davis2019stochastic},~\citep[Section~IV]{li2020understanding}, and the references therein.}

In the process of establishing the said lower 
bounds, which
delineate fundamental limits for subsequent algorithmic developments, 
we develop new machinery that contributes
to second-order variational analysis. 
This includes scalarization
formulae for the Clarke and Goldstein second-order subdifferentials (Proposition~\ref{prop:Clarke-marginal} and Corollary~\ref{cor:Goldstein-marginal}; see, e.g.,~\citep[Proposition~1.50]{mordukhovich2024second} for 
related results), and a second-order Goldstein subdifferential chain rule (Corollary~\ref{cor:2nd-chain-rule}).

The rest of the paper is organized as follows. We first 
lay out the notation and preliminaries in Section~\ref{sec:notation-prelim}. 
Then, we move on to presenting the main results in Section~\ref{sec:main-result}, with proofs deferred to Section~\ref{sec:proofs-main}. Finally, we conclude this paper in Section~\ref{sec:conclusion}.

\section{Notation and preliminaries}\label{sec:notation-prelim}

\subsection{Notation}
The notation in this paper adheres primarily to standard conventions. 
For any $\bx\in\R^n$, its support is defined by $\operatorname{supp}(\bx):=\{i\in[n]:x_i\neq 0\}$, where $[n]:=\{1,\ldots,n\}$. Besides, for any $\bbJ\subseteq[n]$, we use $\bx_{\bbJ}\in\R^{|\bbJ|}$ to denote the subvector of $\bx$ 
consisting of the entries indexed by $\bbJ$.
We use $\bbB^n:=\{\bx\in\R^n:\|\bx\|\le 1\}$ and $\mathbb{S}^n:=\{\bx\in\R^n:\|\bx\|= 1\}$ to denote the unit ball and the unit sphere in $\R^n$, respectively.
For any $\bA\in\R^{m\times n}$, we use 
$\|\bA\|_{\sigma}$ to denote its spectral norm.
For any $\boldsymbol{F}:\R^n\rightarrow\R^m$, we denote by $\operatorname{lip}\boldsymbol{F}(\bx)$ its Lipschitz modulus at $\bx\in\R^n$; see, e.g.,~\citep[Definition~9.1(b)]{rockafellar2009variational}.
For any $f:\R^n\rightarrow\R$, we say that it is $C^k$ for some $k=1,2,\ldots$ if it is $k$-times continuously differentiable.
For any $\mathbb{X}\subseteq\R^n$, we denote by $\operatorname{conv}(\bbX)$ its convex hull.
For $i=1,\ldots,n$, we use $\be^n_i$ to denote the $i$-th standard basis vector in $\R^n$; the superscript $n$ is often omitted when it is clear from the context.
Unless stated otherwise, all limit superiors in this paper are taken in the sense of Painlev{\'e}--Kuratowski; see, e.g.,~\citep[Equation~1.1]{mordukhovich2006variational} and~\citep[Section~5.B]{rockafellar2009variational}.
As an aside, in this paper, we do not distinguish between the maximum and supremum, nor between the minimum and infimum.

\subsection{Generalized differentiation theory and stationarity concepts}\label{sec:diff}
To start, we introduce the (first-order) Clarke and Goldstein subdifferentials.
\begin{definition}\label{def:subdiffs}
Given a 
Lipschitz function $f:\R^n\rightarrow\R$ and a point $\bx\in\R^n$:
\begin{itemize}
    \item The Clarke subdifferential of $f$ at $\bx$ is defined as 
    $$
    \partial_{\dC} f(\bx):=\operatorname{conv}\mleft(\limsup_{\by\rightarrow\bx,\,\by\in\mathbb{D}}\{\nabla f(\by)\}\mright),
    $$
    where $\mathbb{D}\subseteq\R^n$ is the set of points at which $f$ is differentiable;
    see, e.g.,~\citep[Theorem~2.5.1]{clarke1990optimization}.

    \item Given a constant $\delta\ge 0$, the Goldstein $\delta$-subdifferential of $f$ at $\bx$ is defined as
    $$
        \partial_{\delta} f(\bx):=\operatorname{conv}\mleft(\bigcup_{\by\in\bx+\delta\bbB^n}\partial_{\dC} f(\by)\mright);
    $$
    see~\citep[Section~2]{goldstein1977optimization}. 
\end{itemize}
\end{definition}

It is worth noting that for 
a Lipschitz 
function $f$ and a point $\bx$, the subdifferentials
$\partial_{\dC} f(\bx)$ and $\partial_{\delta} f(\bx)$ are always nonempty, convex, and compact; see, e.g.,~\citep[Proposition~2.1.2(a)]{clarke1990optimization} and~\citep[Proposition~2.3]{goldstein1977optimization}.
With the first-order constructions in place, we next introduce tools from second-order variational analysis. We begin with the Clarke generalized Hessian, defined 
analogously
to $\partial_{\dC} f(\bx)$.

\begin{definition}\label{def:second-subdiffs}
Given a $C^{1,1}$-function $f:\R^n\rightarrow\R$ and a point $\bx\in\R^n$:
    \begin{itemize}
        \item The Clarke generalized Hessian of $f$ at $\bx$ is defined as
        $$
            \partial_{\dC}^2 f(\bx):=\operatorname{conv}\mleft(\limsup_{\by\rightarrow\bx,\,\by\in\bbD}\mleft\{\nabla^2 f(\by)\mright\}\mright)\subseteq\R_{\operatorname{sym}}^{n\times n},
        $$
        where $\mathbb{D}\subseteq\R^n$ is the set of points at which $f$ is twice differentiable; see, e.g.,~\citep[Definition~2.6.1]{clarke1990optimization} and~\citep[Theorem~13.52]{rockafellar2009variational}.
    \end{itemize}
\end{definition}

We note that similar to the first-order constructions,
for a $C^{1,1}$-function $f$ and a point $\bx$, the generalized Hessian
$\partial_{\dC}^2 f(\bx)$ is also nonempty, convex, and compact; see, e.g.,~\citep[Proposition~2.6.2(a)]{clarke1990optimization}.
Besides, in alignment with existing second-order subdifferential constructions (see, e.g.,~\citep[Section~1.2]{mordukhovich2024second}), in what follows, we shall identify the set $\partial_{\dC}^2 f(\bx)$ with the set-valued mapping 
$$
\left(\bw\mapsto\mleft\{\bA\bw:\bA\in\partial_{\dC}^2 f(\bx)\mright\}\right):\R^n\rightrightarrows\R^n;
$$
i.e., $\partial_{\dC}^2 f(\bx)(\bw)=\mleft\{\bA\bw:\bA\in\partial_{\dC}^2 f(\bx)\mright\}$ for all $\bw\in\R^n$. 

Very naturally, the construction of $\partial_{\dC}^2 f(\bx)$ gives rise to a second-order optimality condition that parallels the smooth case; i.e., vanishing gradient and positive semidefinite Hessian.

\begin{lemma}[{\citep[Theorem~3.1]{hiriart1984generalized}}]\label{lma:Clarke-SOSP}
    Suppose that $f:\R^n\rightarrow\R$ is $C^{1,1}$, and let $\bx\in\R^n$ be a local minimizer of $f$. Then, we have
    \begin{equation}\label{eq:Clarke-SOSP}
        \nabla f(\bx)=\bd{0}\quad\text{and}\quad\min_{\bw\in\bbS^n}\max_{\bz\in\partial_{\dC}^2 f(\bx)(\bw)}\langle\bz,\bw\rangle\ge 0.
    \end{equation}
\end{lemma}

As an aside,
the following connection between $\partial_{\dC} f(\bx)$ and $\partial_{\dC}^2 f(\bx)$ 
is worth noting.

\begin{proposition}\label{prop:Clarke-marginal}
    Suppose that $f:\R^n\rightarrow\R$ is $C^{1,1}$. Then, we have
    $$
        \partial_{\dC}^2 f(\bx)(\bw)=\partial_{\dC} \langle\bw,\nabla f\rangle (\bx)\quad\text{for all $\bx,\bw\in\R^n$}.
    $$
\end{proposition}

We note 
that this connection is nothing but a scalarization formula for $\partial_{\dC}^2 f(\bx)(\bw)$ that parallels, e.g.,~\citep[Proposition~1.50]{mordukhovich2024second} and~\citep[Proposition~9.24(b)]{rockafellar2009variational}.

\begin{proof}
    This immediately follows by combining~\citep[Theorem~13.52]{rockafellar2009variational} with~\citep[Proposition~1.50]{mordukhovich2024second},~\citep[Proposition~4.3.2(b)]{cui2021modern},~\citep[Exercise~9.9]{rockafellar2009variational}, and the fact that linear maps commute with convex hulls (as can be seen directly from~\citep[Theorem~3.4]{rock1997convex}).
\end{proof}

We now turn to the central object in this 
subsection: The Goldstein second-order $\delta$-subdifferential.

\begin{definition}[Goldstein second-order $\delta$-subdifferential]\label{def:Goldstein-2nd-subdiff}
Given a constant $\delta\ge 0$, a $C^{1,1}$-function $f:\R^n\rightarrow\R$, and a point $\bx\in\R^n$, the Goldstein second-order $\delta$-subdifferential of $f$ at $\bx$ is the set-valued mapping $\partial_{\delta}^2 f(\bx):\R^n\rightrightarrows\R^n$ defined as
    $$
        \partial_{\delta}^2 f(\bx)(\bw):=\operatorname{conv}\mleft(\bigcup_{\by\in\bx+\delta\bbB^n}\partial_{\dC}^2 f(\by)(\bw)\mright)\quad\text{for all $\bw\in\R^n$}.
    $$
\end{definition}

In other words, $\partial_{\delta}^2 f(\bx)(\bw)$ is the smallest convex set that contains all possible 
$\partial_{\dC}^2 f(\by)(\bw)$
from points $\by$ with distance at most $\delta$ from $\bx$, 
in direct analogy with
$\partial_{\delta} f(\bx)$; this explains
why $\partial_{\delta}^2 f(\bx)$ is referred to as ``Goldstein''.
Note that
$\partial_{\delta}^2 f(\bx)(\bw)$ can potentially encompass a broader set of second-order generalized derivatives than $\partial_{\dC}^2 f(\bx)(\bw)$. 
Thus, 
the stationarity concept that 
$\partial_{\delta}^2 f(\bx)$ induces, 
which we introduce in Definition~\ref{def:Goldstein-SOSP}, is much weaker and thus potentially
easier to compute than (\ref{eq:Clarke-SOSP}).
Nonetheless, our main results (Theorem~\ref{thm:lb-1} and Corollary~\ref{cor:lb-1}) show that it
remains deterministically intractable in general.

\begin{definition}[Goldstein approximate second-order stationarity]\label{def:Goldstein-SOSP}
    Given constants $\varepsilon,\eta,\delta\ge 0$, a $C^{1,1}$-function $f:\R^n\rightarrow\R$, 
    and a point $\bx\in\R^n$, we say that $\bx$ is an $(\varepsilon,\eta,\delta)$-Goldstein approximate second-order stationary point of $f$ if 
    $$
        \|\nabla f(\bx)\|\le\varepsilon\quad\text{and}\quad\min_{\bw\in\bbS^n}\max_{\bz\in\partial_{\delta}^2 f(\bx)(\bw)}\langle\bz,\bw\rangle\ge-\eta.
    $$
\end{definition}

As
expected, a scalarization formula for $\partial_{\delta}^2 f(\bx)(\bw)$ can 
be established.

\begin{corollary}\label{cor:Goldstein-marginal}
    Suppose that $\delta\ge 0$ and that $f:\R^n\rightarrow\R$ is $C^{1,1}$. Then, we have
    $$
        \partial_{\delta}^2 f(\bx)(\bw)=\partial_{\delta}\langle\bw,\nabla f\rangle(\bx)\quad\text{for all $\bx,\bw\in\R^n$}.
    $$
\end{corollary}

In particular, Corollary~\ref{cor:Goldstein-marginal} implies that $\partial_{\delta}^2 f(\bx)(\bw)$ is also nonempty, convex, and compact.

\begin{proof}
    Indeed, it directly follows by combining Definition~\ref{def:Goldstein-2nd-subdiff} with Proposition~\ref{prop:Clarke-marginal} and Definition~\ref{def:subdiffs} that
$$
    \partial_{\delta}^2 f(\bx)(\bw)=\operatorname{conv}\mleft(\bigcup_{\by\in\bx+\delta\bbB^n}\partial_{\dC} \langle\bw,\nabla f\rangle (\by)\mright)=\partial_{\delta}\langle\bw,\nabla f\rangle(\bx),
$$
as desired.
\end{proof}

In later developments, we will need a second-order Goldstein subdifferential chain rule. In order to derive this, let us first make explicit a (first-order) Goldstein subdifferential chain rule hidden in the proof of~\citep[Theorem~1]{tian2024no}.

\begin{lemma}\label{lma:tian-chain-rule}
    Let $\delta\ge 0$ be a constant, $f:\R^m\rightarrow\R$ be 
    Lipschitz, and $\bU\in\R^{n\times m}$ be such that $\bU^{\T}\bU=\bI$. Suppose that $g:\R^n\rightarrow\R$ satisfies $g(\bx)=f(\bU^{\T}\bx)$ for all $\bx\in\R^n$. Then, we have
    $$
        \partial_{\delta}g(\bx)=\bU\partial_{\delta}f(\bU^{\T}\bx)\quad\text{for all $\bx\in\R^n$}.
    $$
\end{lemma}

\begin{proof}
    Extend $\bU$ to an orthogonal matrix $\bU^{\prime}\in\R^{n\times n}$, and let $h:\R^n\rightarrow\R$ be such that $h(\bx)=f(\bx_{[m]})$; as a quick sanity check, we have $g(\bx)=h({\bU^{\prime}}^{\T}\bx)$. It follows from the proof of~\citep[Theorem~1]{tian2024no} that $\partial_{\delta}g(\bx)=\bU^{\prime}\partial_{\delta}h({\bU^{\prime}}^{\T}\bx)$.
    However, it also follows from the separable structure of $h$ and~\citep[Proposition~2.5]{rockafellar1985extensions} that $\partial_{\dC}h(\bx)=\partial_{\dC}f(\bx_{[m]})\times\{\bd{0}\}$. As a result, we have
    $$
    \begin{aligned}
        \partial_{\delta}h({\bU^{\prime}}^{\T}\bx)&=\operatorname{conv}\mleft(\bigcup_{\by\in{\bU^{\prime}}^{\T}\bx+\delta\bbB^n}\partial_{\dC} h(\by)\mright)=\operatorname{conv}\mleft(\bigcup_{\by\in{\bU^{\prime}}^{\T}\bx+\delta\bbB^n}
        \begin{pmatrix}
            \partial_{\dC}f(\by_{[m]}) \\
            \bd{0}
        \end{pmatrix}
        \mright)\\
        &=
        \begin{pmatrix}
            \operatorname{conv}\mleft(\bigcup_{\bz\in({\bU^{\prime}}^{\T}\bx)_{[m]}+\delta\bbB^m}\partial_{\dC}f(\bz)\mright) \\
            \bd{0}
        \end{pmatrix}
        =
        \begin{pmatrix}
            \partial_{\delta}f\bigl(({\bU^{\prime}}^{\T}\bx)_{[m]}\bigr) \\
            \bd{0}
        \end{pmatrix}
        =
        \begin{pmatrix}
            \partial_{\delta}f(\bU^{\T}\bx) \\
            \bd{0}
        \end{pmatrix},
    \end{aligned}
    $$
    where the third equality follows from the fact that
    $$
    \begin{aligned}
        \mleft\{\by_{[m]}:\by\in{\bU^{\prime}}^{\T}\bx+\delta\bbB^n\mright\}&=\mleft\{\bz:\left\|\begin{pmatrix}
            \bz \\
            \bz^{\prime}
        \end{pmatrix}-{\bU^{\prime}}^{\T}\bx\right\|\le\delta~\textnormal{for some $\bz^{\prime}\in\R^{n-m}$}\mright\}\\
        &=\mleft\{\bz:\min_{\bz^{\prime}\in\R^{n-m}}\left\|\begin{pmatrix}
            \bz \\
            \bz^{\prime}
        \end{pmatrix}-{\bU^{\prime}}^{\T}\bx\right\|\le\delta\mright\}\\
        &=\mleft\{\bz:\|\bz-({\bU^{\prime}}^{\T}\bx)_{[m]}\|\le\delta\mright\}=({\bU^{\prime}}^{\T}\bx)_{[m]}+\delta\bbB^m.
    \end{aligned}
    $$
    This, together with $\partial_{\delta}g(\bx)=\bU^{\prime}\partial_{\delta}h({\bU^{\prime}}^{\T}\bx)$ mentioned earlier, immediately implies that
    $$
        \partial_{\delta}g(\bx)=\bU^{\prime}\begin{pmatrix}
            \partial_{\delta}f(\bU^{\T}\bx) \\
            \bd{0}
        \end{pmatrix}=\bU\partial_{\delta}f(\bU^{\T}\bx),
    $$
    as desired. 
\end{proof}

We remark that Lemma~\ref{lma:tian-chain-rule} is slightly stronger than the chain rule implicit in the proof of~\citep[Theorem~1]{tian2024no}, as here $\bU$ is only required to have orthonormal columns, rather than to be orthogonal.
With Lemma~\ref{lma:tian-chain-rule} in place, we are now ready to derive the second-order chain rule.
\begin{corollary}\label{cor:2nd-chain-rule}
    Let $\delta\ge 0$ be a constant, $f:\R^m\rightarrow\R$ be $C^{1,1}$, and $\bU\in\R^{n\times m}$ be such that $\bU^{\T}\bU=\bI$. Suppose that $g:\R^n\rightarrow\R$ satisfies $g(\bx)=f(\bU^{\T}\bx)$ for all $\bx\in\R^n$. Then, we have
    $$
        \partial_{\delta}^2 g(\bx)(\bw)=\bU\partial_{\delta}^2 f(\bU^{\T}\bx)(\bU^{\T}\bw)\quad\text{for all $\bx,\bw\in\R^n$}.
    $$
\end{corollary}

\begin{proof}
    To start, we observe through the smooth chain rule 
    that for all $\bx,\bw\in\R^n$,
    $$
        \langle\bw,\nabla g\rangle(\bx)=\langle\bw,\nabla g(\bx)\rangle=\langle\bw,\bU\nabla f(\bU^{\T}\bx)\rangle=\langle\bU^{\T}\bw,\nabla f(\bU^{\T}\bx)\rangle=\langle\bU^{\T}\bw,\nabla f\rangle(\bU^{\T}\bx).
    $$
    Besides, as 
    $f$ is $C^{1,1}$
    as assumed, we know 
    that $\langle\bU^{\T}\bw,\nabla f\rangle$ is 
    Lipschitz continuous and 
    $g$ is $C^{1,1}$.
    It then follows by combining the above facts with Corollary~\ref{cor:Goldstein-marginal} and Lemma~\ref{lma:tian-chain-rule} that
    $$
        \partial_{\delta}^2 g(\bx)(\bw)=\partial_{\delta}\langle\bw,\nabla g\rangle(\bx)=\bU\partial_{\delta}\langle\bU^{\T}\bw,\nabla f\rangle(\bU^{\T}\bx)=\bU\partial_{\delta}^2 f(\bU^{\T}\bx)(\bU^{\T}\bw),
    $$
    as desired. This completes the proof.
\end{proof}

We note in closing this 
subsection
that second-order variational analysis is a rich and 
thriving
subject that encompasses concepts, theories, and applications far beyond the preceding discussions. 
We refer interested readers to the recent book by~\citet{mordukhovich2024second}. 

\subsection{Oracles, algorithm classes, and function classes}
\subsubsection{Oracles}
In this paper, unless stated otherwise, we primarily restrict our attention to second-order oracles.
We say that $\mathbfcal{O}:\R^n\rightarrow\R\times\R^n\times2^{\R_{\operatorname{sym}}^{n\times n}}$ is the second-order oracle that corresponds to a $C^{1,1}$-function $f:\R^n\rightarrow\R$ if $\mathbfcal{O}(\bx)=\bigl(f(\bx),\nabla f(\bx),\partial_{\dC}^2 f(\bx)\bigr)$ for all $\bx\in\R^n$.
With a slight abuse of notation, we write $\mathbfcal{O}(\bx)=\bigl(f(\bx),\nabla f(\bx),\nabla^2 f(\bx)\bigr)$ when $f$ is 
$C^2$
near $\bx$ (even if $f$ may fail to be (globally) $C^{1,1}$), identifying the singleton set $\partial_{\dC}^2 f(\bx)$ with its unique element $\nabla^2 f(\bx)$.


\subsubsection{Algorithm classes}\label{sec:alg}
In this paper, an algorithm that seeks solutions over $\R^n$ (for a general oracle $\mathbfcal{O}:\R^n\rightarrow\bbO$ where $\bbO$ is an arbitrary set) consists of a (possibly randomized) initial point $\bx_0\in\R^n$ and a sequence of (possibly randomized) mappings $\bA_k:(\R^n\times\bbO)^{k+1}\rightarrow\R^n$ for $k=0,1,\ldots$. The sequence of iterates generated by the algorithm is defined recursively by
\begin{equation*}\label{eq:seq-alg}
    \bx_{k+1}:=\bA_k\bigl(\bx_0,\mathbfcal{O}(\bx_0),\ldots,\bx_k,\mathbfcal{O}(\bx_k)\bigr)\quad\text{for $k=0,1,\ldots$}.
\end{equation*}
With a slight abuse of terminology, we will occasionally call $\bx_0$ an iterate as well; i.e., $\bx_0,\bx_1,\ldots$ are collectively referred to as the iterates generated by the algorithm.
We say that an algorithm is deterministic if its constituents $\bx_0$ and $\bA_0,\bA_1,\ldots$ are so. 



As we shall only consider functions that are $C^2$ near the queried points (see Section~\ref{sec:fcn-class}), in what follows, it suffices to consider algorithms that access Hessians rather than Clarke generalized Hessians. For this purpose, let us denote by
$$
    \bbA_{\ddet}^n(m):=
    \left\{
    (\bx_0,\bA_0,\ldots,\bA_{m-1}):
    \textnormal{$\bx_0\in\R^n$ and $\bA_k:(\R^n\times\R\times\R^n\times\R_{\operatorname{sym}}^{n\times n})^{k+1}\rightarrow\R^n$}
    \right\}
$$
the set of all second-order deterministic algorithms that seek solutions over $\R^n$ with $m\ge 1$ iterations.

With the above preparations in place, we turn to recall the notion of second-order zero-respecting algorithms~\citep[Equation~5]{carmon2020lower}, adapted to our setting.
In brief, these algorithms 
can only update coordinates that have been revealed as relevant through prior oracle calls.

\begin{definition}\label{def:zero-resp}
    We say that 
    $\mathbfcal{A}=(\bx_0,\bA_0,\ldots,\bA_{m-1})\in\bbA_{\ddet}^n(m)$ is second-order zero-respecting if for every $f:\R^n\rightarrow\R$ that is $C^2$ near
    the iterates $\bx_0,\ldots,\bx_{m-1}$ generated by $\mathbfcal{A}$ when applied to (the second-order oracle that corresponds to) $f$, it holds for all $k=0,\ldots,m$ that
    \begin{equation}\label{eq:zero-resp}
        \operatorname{supp}(\bx_k)\subseteq\left(\bigcup_{0\le\ell<k}\operatorname{supp}\bigl(\nabla f(\bx_{\ell})\bigr)\right)\cup\left(\bigcup_{0\le\ell<k}\bigl\{i\in[n]:(\nabla^2 f(\bx_{\ell}))_{i}\neq\bd{0}\bigr\}\right).
    \end{equation}
    The set of all such algorithms is denoted by $\bbA_{\ddetzr}^n(m)$.
\end{definition}

It is worth noting that by letting $k=0$ in (\ref{eq:zero-resp}), we obtain $\bx_0=\bd{0}$; in other words, second-order deterministic zero-respecting algorithms are necessarily initialized at the origin.

We close this 
part
with a discussion of the scope of the algorithm class $\bbA_{\ddetzr}^n(m)\subseteq\bbA_{\ddet}^n(m)$.
As the functions we consider are indistinguishable from $C^2$-functions near the queried points,
different second-order subdifferential constructions all collapse to the Hessian at the queried points
(see, e.g.,~\citep[Proposition~1.74]{mordukhovich2024second}), and methods for smooth optimization are also applicable.
As a consequence,
the class of second-order deterministic zero-respecting algorithms (i.e., $\bbA_{\ddetzr}^n(m)$) covers a wide array of existing smooth and nonsmooth methods.
Concretely, 
this includes,
e.g.,
the gradient method~\citep[Section~1.2.3]{nesterov2018lectures}, Nesterov's accelerated gradient method~\citep{nesterov1983method,su2016differential} (as long as the extrapolated points are also treated as part of the iterates), the conjugate gradient method~\citep[Algorithm~5.4]{nocedal2006numerical} (as long as the step sizes are, e.g., determined \textit{a priori}), Newton's method~\citep[Section~1.2.4]{nesterov2018lectures} (as long as the Newton direction is taken as $-\nabla^2 f(\bx)^{\dagger}\nabla f(\bx)$ when $\nabla^2 f(\bx)$ is singular\footnote{We note in passing that this direction corresponds to the minimum-norm solution of the linear system $\nabla^2 f(\bx)\bdd+\nabla f(\bx)=\bd{0}$ (when it is satisfiable).}), nonsmooth Newton methods based on Bouligand generalized Hessians~\citep{qi1993convergence}, directional derivatives~\citep{pang1990newton}, graphical derivatives~\citep{hoheisel2012generalized,mordukhovich2021generalized}, Mordukhovich coderivatives~\citep{mordukhovich2021generalized,khanh2023generalized}, and Clarke generalized Hessians~\citep{kummer1988newton,qi1993nonsmooth} (as long as the same convention for choosing the Newton direction is adopted), the BFGS method~\citep[Algorithm~6.1]{nocedal2006numerical} (as long as the initial inverse Hessian approximation is, e.g., diagonal, and the step sizes are, e.g., determined \textit{a priori}), the cubic Newton method~\citep{nesterov2006cubic} (as long as the (cubic) regularization weights are, e.g., determined \textit{a priori}), the trust-region method~\citep[Algorithm~4.1]{nocedal2006numerical} (as long as 
the trust-region step is taken as $-\nabla^2 f(\bx)^{\dagger}\nabla f(\bx)$ when 
the trust-region constraint is redundant,
and the trial points are also treated as part of the iterates),
and any deterministic 
second-order
method that satisfies the so-called linear span assumption~\citep[Assumption~2.1.4]{nesterov2018lectures}, assuming that they are all initialized at the origin.
In particular, this implies that the hardness result in Theorem~\ref{thm:lb-1} 
applies to all of these methods.

\subsubsection{Function classes}\label{sec:fcn-class}
In this paper, we focus exclusively on subclasses of $C^{1,1}$-functions that are bounded from below. To ease the exposition, we introduce for all $\mathbfcal{A}=(\bx_0,\bA_0,\ldots,\bA_{m-1})\in\bbA_{\ddet}^n(m)$ the shorthand
$$
    \bbF_{\!C^1}^n(L,G;\mathbfcal{A}):=\left\{f:\R^n\rightarrow\R:
    \begin{gathered}
    \textnormal{$f$ is $C^1$ with $\nabla f$ being $L$-Lipschitz continuous}\\
    f(\bx_0)-\min\{f(\bx):\bx\in\R^n\}\le G\\
    \textnormal{$f$ is $C^2$ around the iterates $\bx_0,\ldots,\bx_m$ generated by $\mathbfcal{A}$}
    \end{gathered}
    \right\}.
$$
We remark that imposing the $C^2$-smoothness of $f$ around $\bx_0,\ldots,\bx_m$ is solely to expand the range of applicable algorithms; see the discussion at the end of Section~\ref{sec:alg}.
Besides, as this additional restriction only shrinks the function class, the lower bounds in Theorem~\ref{thm:lb-1} and Corollary~\ref{cor:lb-1} carry over \textit{a fortiori} to the larger class of all lower-bounded $C^{1,1}$-functions.

\section{Main results}\label{sec:main-result}

\subsection{Lower bound for second-order deterministic zero-respecting algorithms}
Our first lower bound,
the impossibility of 
computing Goldstein approximate second-order stationary points of $C^{1,1}$-functions in finite time by
algorithms in $\bbA_{\ddetzr}^n(m)$, is as follows.
\begin{theorem}\label{thm:lb-1}
    We have
    $$
        \max_{\substack{m=1,2,\ldots\\
        n=2,3,\ldots\\
        \mathbfcal{A}\in\bbA_{\ddetzr}^n(m)}}\min_{f\in\bbF_{\!C^1}^n(4.086,17/256;\mathbfcal{A})}\max_{k=0,\ldots,m}\min_{\bw\in\bbS^n}\max_{\bz\in\operatorname{conv}\mleft(\bigcup_{\by\in\bx_k+\bbB^n/8}\partial_{\dC}^2 f(\by)(\bw)\mright)}\langle\bz,\bw\rangle\le -0.0092714,
    $$
    where 
    $\bx_0,\ldots,\bx_m$
    are the iterates generated by $\mathbfcal{A}$ when applied to (the second-order oracle that corresponds to) $f$; see Section~\ref{sec:alg}.
\end{theorem}

In 
particular, Theorem~\ref{thm:lb-1} asserts that for any $m\ge 1$, any $n\ge 2$, and any $\mathbfcal{A}\in\bbA_{\ddetzr}^n(m)$, there exists an $f\in\bbF_{\!C^1}^n(4.086,17/256;\mathbfcal{A})$,
such that the iterates $\bx_0,\ldots,\bx_m$ generated by $\mathbfcal{A}$ satisfy
\begin{equation*}
    \min_{\bw\in\bbS^n}\max_{\bz\in\partial_{1/8}^2 f(\bx_k)(\bw)}\langle\bz,\bw\rangle=\min_{\bw\in\bbS^n}\max_{\bz\in\operatorname{conv}\mleft(\bigcup_{\by\in\bx_k+\bbB^n/8}\partial_{\dC}^2 f(\by)(\bw)\mright)}\langle\bz,\bw\rangle\le -0.0092714
\end{equation*}
for all $k=0,\ldots,m$;
i.e., 
none of the iterates
$\bx_0,\ldots,\bx_m$
is $(\varepsilon,\eta,\delta)$-Goldstein approximate second-order stationary for all $\varepsilon\ge 0$, $\eta\in [0,0.0092714)$, and $\delta\in[0,1/8]$.

The proof of Theorem~\ref{thm:lb-1} is (a bit involved and thus) deferred to 
Section~\ref{sec:proof-theorem} and Appendix~\ref{app:proofs-main}.
\subsection{Lower bound for second-order (general) deterministic algorithms}
While Theorem~\ref{thm:lb-1} only applies to 
$\bbA_{\ddetzr}^n(m)$, with the help of an orthogonal transformation as employed in, e.g.,~\citep[Section~7.2.2]{nemirovskij1983problem} and the proof of~\citep[Theorem~1]{tian2024no}, it can be further extended to 
$\bbA_{\ddet}^n(m)$, albeit with a slight reduction in hardness; i.e., from finite-time 
intractability
to dimension-free 
intractability.

\begin{corollary}\label{cor:lb-1}
    We have
    $$
        \max_{\substack{m=1,2,\ldots\\
        n=m+3,m+4,\ldots\\
        \mathbfcal{A}\in\bbA_{\ddet}^n(m)}}\min_{f\in\bbF_{\!C^1}^n(4.086,17/256;\mathbfcal{A})}\max_{k=0,\ldots,m}\min_{\bw\in\bbS^n}\max_{\bz\in\operatorname{conv}\mleft(\bigcup_{\by\in\bx_k+\bbB^n/8}\partial_{\dC}^2 f(\by)(\bw)\mright)}\langle\bz,\bw\rangle\le -0.0092714,
    $$
    where 
    $\bx_0,\ldots,\bx_m$
    are the iterates generated by $\mathbfcal{A}$ when applied to (the second-order oracle that corresponds to) $f$; see Section~\ref{sec:alg}.
\end{corollary}

The proof of Corollary~\ref{cor:lb-1}, which relies on the second-order Goldstein subdifferential chain rule established in Section~\ref{sec:diff} (i.e., Corollary~\ref{cor:2nd-chain-rule}), is deferred to 
Section~\ref{sec:proof-corollary}.


\section{Proofs of main results}\label{sec:proofs-main}
\subsection{Proof of Theorem~\ref{thm:lb-1}}\label{sec:proof-theorem}
In this subsection, we 
outline the main steps of the proof of
Theorem~\ref{thm:lb-1} 
and defer
the proofs of the auxiliary results used along the way to Appendix~\ref{app:proofs-main}, where a complete and self-contained proof
of Theorem~\ref{thm:lb-1}
is also provided.
In what follows, we shall restrict attention to the case $n=2$; for $n\ge 3$, it suffices to extend the construction to be given for $n=2$ in the canonical way.
Let $m\ge 1$ be arbitrary.

To begin, consider the resisting oracle $\mathbfcal{O}:\R^2\rightarrow\R\times\R^2\times\R_{\operatorname{sym}}^{2\times 2}$ defined for all $\bx\in\R^2$ by $\mathbfcal{O}(\bx):=(0,\bd{0},-\be_1\be_1^{\T})$
and an arbitrary $\mathbfcal{A}\in\bbA_{\ddetzr}^2(m)$;
here, $\mathbfcal{O}$ is referred to as ``resisting'' in that $-\be_1\be_1^{\T}$ is not positive semidefinite.
Let $\bx_0,\ldots,\bx_m$ be the iterates generated by $\mathbfcal{A}$ after $m$ rounds of interaction with $\mathbfcal{O}$. 
By 
the zero-respecting property of $\mathbfcal{A}$,
we have
$\bx_0=\bd{0}$ 
and 
$\operatorname{supp}(\bx_k)\subseteq \{1\}$ for all $k=0,\ldots,m$; i.e., $\bx_0,\ldots,\bx_m$ lie on the $x$-axis.
Without loss of generality, we may assume that $\bx_i\neq \bx_j$ for all $i\neq j$, and, by relabeling (and abusing notation), that $\bx_0,\ldots,\bx_m$ are sorted in ascending order of their $x$-coordinates.
In particular, this implies that $\bx_0$ may no longer be $\bd{0}$ anymore.
For ease of exposition, we also define $\bx_{-1}:=(-\infty,0)^{\T}$ and $\bx_{m+1}=(\infty,0)^{\T}$, and with a further 
abuse of notation, we denote by $x_k$ the $x$-coordinate of $\bx_k$ for $k=-1,\ldots,m+1$; e.g., $x_{-1}=-\infty$.
Inspired by existing constructions~\citep{tian2024no,jordan2023deterministic} that lower bound the complexity of Lipschitz optimization, in what follows, we shall construct a hard function in
$\bbF_{\!C^1}^2(4.086,17/256;\mathbfcal{A})$
such that
it is compatible with 
$\mathbfcal{O}$ 
along 
$\bx_0,\ldots,\bx_m$ 
but 
none of the iterates
is Goldstein approximate second-order stationary 
up to the tolerances specified in
Theorem~\ref{thm:lb-1}.

Define
$$
    \zeta:=\min\{\xi,1\}\in(0,1],\quad\text{where $\xi:=\min\bigl\{|x_i-x_j|:i,j=0,\ldots,m~\text{with}~i\neq j\bigr\}>0$}.
$$
Consider the hard function $f(\bullet,\bullet;\zeta):\R^2\rightarrow\R$ given by
\begin{equation}
\label{main-eq:def-f}
    f(x,y;\zeta):=
    \begin{dcases}
        f_k(x,y;\zeta), & |x-x_k|\le \zeta/4,\\
        q(y;1), & \textnormal{otherwise},
    \end{dcases}
\end{equation}
where for $k=0,\ldots,m$, each $\delta,\gamma>0$, and $(\bullet)_+ := \max\{\bullet,0\}$,
$$
    f_k(x,y;\zeta):=p_2(x-x_k,y;\zeta/32)\cdot\left(-\frac{1}{2} (x-x_k)^2\right)+\bigl(1-p_2(x-x_k,y;\zeta/32)\bigr)\cdot h(x-x_k,y;\zeta/4),
$$
$$
    p_2(x,y;\delta):=
    \left(1-\frac{x^2+y^2}{\delta^{2}}\right)_+^{4},
    \quad
    q(x;\gamma):=
     -\frac{1}{2}x^{2}
+\left(|x|-\frac{\gamma}{4}\right)_+^{2}
-\left(|x|-\frac{3\gamma}{4}\right)_+^{2}
+\frac{1}{2}(|x|-\gamma)_+^{2},
$$
and $h(x,y;\gamma):=q(x;\gamma)+q(y;1)$.
We remark that $p_2(\bullet,\bullet;\delta)$ is $C^3$, whereas $q(\bullet;\gamma)$ and $h(\bullet,\bullet;\gamma)$ are $C^{1,1}$. 
As 
\begin{equation*}
    h(x,y;\gamma)=-\frac{1}{2} x^2-\frac{1}{2} y^2=-2^{-1}\mleft\|\bullet\mright\|^2 (x,y)\quad\text{for all $(x,y)^{\T}\in\left[-\frac{\gamma}{4},\frac{\gamma}{4}\right]\times\left[-\frac{1}{4},\frac{1}{4}\right]$},
\end{equation*}
we see that 
around
each $\bx_k$,
the function $f(\bullet,\bullet;\zeta)$ is defined to 
coincide with the strongly concave function $-2^{-1}\mleft\|\bullet\mright\|^2$ (upon recentering at $\bx_k$), modulo the local surgery induced by 
$p_2(\bullet,\bullet;\delta)$ that smoothly interpolates between $-2^{-1}\mleft\|\bullet\mright\|^2$ and $(x,y)\mapsto -2^{-1} x^2$ (that corresponds to $\mathbfcal{O}$) and yields the $C^3$-blend $g(\bullet,\bullet;\delta):\R^2\rightarrow\R$ (which is also $C^{1,1}$; see Proposition~\ref{maintext-prop:useful-properties}) defined by
$$
    g(x,y;\delta):= p_2(x,y;\delta)\cdot\left(-\frac{1}{2} x^2\right)+\bigl(1-p_2(x,y;\delta)\bigr)\cdot\left(-\frac{1}{2} x^2-\frac{1}{2} y^2\right).
$$ 
By contrast,
away from $\bx_0,\ldots,\bx_m$, the function $f(\bullet,\bullet;\zeta)$ is interpolated such that
it remains $C^{1,1}$, becomes lower-bounded, and features a sufficiently wide infinite horizontal strip around the $x$-axis with strong 
stationarity refutation;
see Corollary~\ref{maintext-cor:all-desired}.
It is worth noting
that the purpose of carrying out the local surgeries is to 
give rise to
the 
compatibility of $f(\bullet,\bullet;\zeta)$ with $\mathbfcal{O}$ along $\bx_0,\ldots,\bx_m$, which is otherwise absent as $\nabla^2 (-2^{-1}\mleft\|\bullet\mright\|^2)(\bx) =-\bI\neq-\be_1\be_1^{\T}$,
while 
preserving the essential
smoothness and refutation properties of 
$-2^{-1}\mleft\|\bullet\mright\|^2$.

Before analyzing $f(\bullet,\bullet;\zeta)$, we first record some useful properties of $g(\bullet,\bullet;\delta)$ and $h(\bullet,\bullet;\gamma)$.

\begin{proposition}\label{maintext-prop:useful-properties}
    Let $\delta,\gamma>0$ be constants.
    The following statements hold.
    \begin{itemize}
        \item \textbf{(Stationarity refutations)} 
        Let
        $\bw:=(10,11)^{\T}/\sqrt{221}\in\bbS^{2}$. We have 
        $$
            \langle\nabla^2 g(x,y;\delta),\bw\bw^{\T}\rangle\le-0.0092714\quad\text{for all $x,y\in\R$},
        $$
        $$
            \max\bigl\{\langle\bz,\bw\rangle:\bz\in\partial_{\dC}^2 h(x,y;\gamma)(\bw)\bigr\}\le-0.095\quad\text{for all $(x,y)^{\T}\in\R\times\left(-\frac{1}{4},\frac{1}{4}\right)$}.
        $$
        \item \textbf{($C^{1,1}$-smoothness)} We have $\operatorname{lip}\nabla g(x,y;\delta)\le 4.086$ and $\operatorname{lip}\nabla h(x,y;\gamma)\le 1$ for all $x,y\in\R$.

        \item \textbf{(Oracle compatibility)} We have $g(0,0;\delta)=0$, $\nabla g(0,0;\delta)=\bd{0}$, and $\nabla^2 g(0,0;\delta)=-\be_1\be_1^{\T}$.
    \end{itemize}
\end{proposition}

We 
now proceed to analyze
$f(\bullet,\bullet;\zeta)$,
starting
with a basic observation on the piecewise structure of $f_k(\bullet,\bullet;\zeta)$
in terms of
$g(\bullet,\bullet;\delta)$ and $h(\bullet,\bullet;\gamma)$.
\begin{lemma}
\label{maintext-lma:trivial-obs}
    For
    all $k=0,\ldots,m$, we have 
    \begin{equation*}
        f_k(x,y;\zeta)=
        \begin{dcases}
            h(x-x_k,y;\zeta/4), & (x,y)^{\T}\notin (x_k,0)^{\T}+\frac{\zeta}{32}\bbB^2,\\
            g(x-x_k,y;\zeta/32), & (x,y)^{\T}\in\left[x_k-\frac{\zeta/4}{4},x_k+\frac{\zeta/4}{4}\right]\times\left[-\frac{1}{4},\frac{1}{4}\right].
        \end{dcases}
    \end{equation*}
\end{lemma}

It is worth noting that the cases in Lemma~\ref{maintext-lma:trivial-obs} are collectively exhaustive, as
\begin{equation*}
    (x_k,0)^{\T}+\frac{\zeta}{32}\bbB^2\subseteq(x_k,0)^{\T}+\frac{1}{2}\cdot\left(\left[-\frac{\zeta/4}{4},\frac{\zeta/4}{4}\right]\times\left[-\frac{1}{4},\frac{1}{4}\right]\right).
\end{equation*}
With Lemma~\ref{maintext-lma:trivial-obs} in place, we are now 
ready to derive a piecewise representation of $f(\bullet,\bullet;\zeta)$.

\begin{proposition}
\label{maintext-prop:covering}
    Consider the open sets $\bbH_0,\ldots,\bbH_m$ and $\bbP_0,\ldots,\bbP_m$ defined by
    $$
        \bbH_k:=\left(\left(x_{k-1}+\frac{\zeta}{4},x_{k+1}-\frac{\zeta}{4}\right)\times\R\right)\cap\bbB_k^{\complement}
        \quad\text{and}\quad
        \bbP_k:=\left(x_k-\frac{\zeta/4}{4},x_k+\frac{\zeta/4}{4}\right)\times\left(-\frac{1}{4},\frac{1}{4}\right),
    $$
    where $\bbB_k:=(x_k,0)^{\T}+{\zeta\bbB^2}/{32}$. 
    We have
    $\bigcup_{k=0}^m\bbH_k\cup\bigcup_{k=0}^m\bbP_k=\R^2$ and
    $$
        f(x,y;\zeta)=
        \begin{dcases}
            h(x-x_k,y;\zeta/4), & (x,y)^{\T}\in\bbH_k,\\
            g(x-x_k,y;\zeta/32), & (x,y)^{\T}\in\bbP_k.
        \end{dcases}
    $$
\end{proposition}


By combining Proposition~\ref{maintext-prop:useful-properties} and Proposition~\ref{maintext-prop:covering} (and carrying out a few further steps), 
we have the following
properties of $f(\bullet,\bullet;\zeta)$. 

\begin{corollary}\label{maintext-cor:all-desired}
    The following statements hold.
    \begin{itemize}
        \item \textbf{($C^{1,1}$-smoothness)} The function $f(\bullet,\bullet;\zeta)$ is $C^{1}$ with $\nabla f(\bullet,\bullet;\zeta)$ being $4.086$-Lipschitz. 
        
        \item \textbf{(Lower boundedness)} We have $f(0,0;\zeta)-\min\{f(x,y;\zeta):x,y\in\R\}\le 17/256$.

        \item \textbf{(Local $C^{3}$-smoothness)} The function $f(\bullet,\bullet;\zeta)$ is $C^3$ (and thus $C^2$) around $\bx_0,\ldots,\bx_m$.
    
        \item \textbf{(Oracle compatibility)} The function $f(\bullet,\bullet;\zeta)$ is compatible with $\mathbfcal{O}$ along $\bx_0,\ldots,\bx_m$.

        \item \textbf{(Band refutation)} 
        Let
        $\bw:=(10,11)^{\T}/\sqrt{221}\in\bbS^{2}$. We have 
    $$
        \max\mleft\{\langle\bz,\bw\rangle:\bz\in\operatorname{conv}\mleft(\bigcup_{(x,y)^{\T}\in\R\times(-1/4,1/4)}\partial_{\dC}^2 f(x,y;\zeta)(\bw)\mright)\mright\}\le -0.0092714.
    $$
    \end{itemize}
\end{corollary}

We remark that the first three properties 
together
imply that $f(\bullet,\bullet;\zeta)\in\bbF_{\!C^1}^2(4.086,17/256;\mathbfcal{A})$.
By combining this membership, the other two properties in Corollary~\ref{maintext-cor:all-desired}, and the fact that
$$
    \operatorname{conv}\mleft(\bigcup_{(x,y)^{\T}\in(x_k,0)^{\T}+\bbB^2/8}\partial_{\dC}^2 f(x,y;\zeta)(\bw)\mright)\subseteq\operatorname{conv}\mleft(\bigcup_{(x,y)^{\T}\in\R\times(-1/4,1/4)}\partial_{\dC}^2 f(x,y;\zeta)(\bw)\mright)
$$
for all $\bw\in\R^2$ and $k=0,\ldots,m$,
the whole proof is complete.

\subsection{Proof of Corollary~\ref{cor:lb-1}}\label{sec:proof-corollary}

\begin{myproof}{Corollary~\ref{cor:lb-1}}
    Let $m\ge 1$, $n\ge m+3$, and $\mathbfcal{A}\in\bbA_{\ddet}^n(m)$ be arbitrary. Consider the resisting oracle $\mathbfcal{O}:\R^n\rightarrow\R\times\R^n\times\R_{\operatorname{sym}}^{n\times n}$ given by $\mathbfcal{O}(\bx):=(0,\bd{0},-\be_1\be_1^{\T})$ for all $\bx\in\R^n$. Let $\bx_0,\ldots,\bx_m$ be the iterates generated by $\mathbfcal{A}$ after $m$ rounds of interaction with $\mathbfcal{O}$, and $f_2:\R^2\rightarrow\R$ be the function constructed in (\ref{main-eq:def-f}) with 
    $(\be_1^{\T}\bx_0,0)^{\T},\ldots,(\be_1^{\T}\bx_m,0)^{\T}$ as the iterates.
    In addition, let $\bv\in\bbS^n\cap\operatorname{span}\{\be_1,\bx_0,\ldots,\bx_m\}^{\perp}$ be arbitrary; this is possible as $n\ge m+3$. As a quick sanity check, 
    for $\bU:=(\be_1,\bv)\in\R^{n\times 2}$, we have 
    $(\be_1^{\T}\bx_k,0)^{\T}=\bU^{\T}\bx_k$ for all $k=0,\ldots,m$.
    We also note that $\bU^{\T}\bU=\bI$. In the sequel, similar to Section~\ref{sec:proof-theorem}, we shall show that the function $f:\R^n\rightarrow\R$ defined by $f(\bx):=f_2(\bU^{\T}\bx)$ is compatible with $\mathbfcal{O}$ along $\bx_0,\ldots,\bx_m$ and possesses all the desired properties stated in Corollary~\ref{cor:lb-1}.

    We begin with the oracle compatibility of $f$ along $\bx_0,\ldots,\bx_m$. 
    Indeed, it follows 
    by combining 
    the local $C^{3}$-smoothness
    and oracle compatibility 
    of $f_2$
    in Corollary~\ref{maintext-cor:all-desired} with the smooth chain rule 
    that
    $$
        \begin{dcases}
            f(\bx_k)=f_2(\bU^{\T}\bx_k)=0,\\
            \nabla f(\bx_k)=\bU\nabla f_2(\bU^{\T}\bx_k)=\bU\bd{0}=\bd{0},\\
            \nabla^2 f(\bx_k)=\bU\nabla^2 f_2(\bU^{\T}\bx_k)\bU^{\T}=\bU(-\be_1\be_1^{\T})\bU^{\T}=-\be_1\be_1^{\T},
        \end{dcases}
        \quad\text{for all $k=0,\ldots,m$},
    $$
    as desired.

    We next verify that $f\in\bbF_{\!C^1}^n(4.086,17/256;\mathbfcal{A})$. To see the first property, it suffices to observe through the $C^{1,1}$-smoothness of $f_2$ in Corollary~\ref{maintext-cor:all-desired} and the smooth chain rule that
    $$
    \begin{aligned}
        \|\nabla f(\bx)-\nabla f(\by)\|&=\|\bU\nabla f_2(\bU^{\T}\bx)-\bU\nabla f_2(\bU^{\T}\by)\|=\|\nabla f_2(\bU^{\T}\bx)-\nabla f_2(\bU^{\T}\by)\|\\
        &\le4.086\cdot\|\bU^{\T}\bx-\bU^{\T}\by\|\le4.086\cdot\|\bx-\by\|\quad\text{for all $\bx,\by\in\R^n$};
    \end{aligned}
    $$
    the other two properties are trivial.

    Finally, we deal with stationarity refutations. To start, we observe 
    through Corollary~\ref{cor:2nd-chain-rule} that for all $\bw\in\R^n$ and $k=0,\ldots,m$,
    $$
        \operatorname{conv}\mleft(\bigcup_{\by\in\bx_k+\bbB^n/8}\partial_{\dC}^2 f(\by)(\bw)\mright)
        =\bU\operatorname{conv}\mleft(\bigcup_{\by\in\bU^{\T}\bx_k+\bbB^2/8}\partial_{\dC}^2 f_2(\by)(\bU^{\T}\bw)\mright).
    $$
    As a result, 
    for $\bw:=\bU\bw_0\in\bbS^{n}$ 
    with
    $\bw_0:=(10,11)^{\T}/\sqrt{221}\in\bbS^{2}$, we have 
    $$
    \begin{aligned}
        \max_{\bz\in\operatorname{conv}\mleft(\bigcup_{\by\in\bx_k+\bbB^n/8}\partial_{\dC}^2 f(\by)(\bw)\mright)}\langle\bz,\bw\rangle&=\max_{\bz\in\bU\operatorname{conv}\mleft(\bigcup_{\by\in\bU^{\T}\bx_k+\bbB^2/8}\partial_{\dC}^2 f_2(\by)(\bU^{\T}\bw)\mright)}\langle\bz,\bw\rangle\\
        &=\max_{\bz\in\operatorname{conv}\mleft(\bigcup_{\by\in\bU^{\T}\bx_k+\bbB^2/8}\partial_{\dC}^2 f_2(\by)(\bU^{\T}\bw)\mright)}\langle\bz,\bU^{\T}\bw\rangle\\
        &=\max_{\bz\in\operatorname{conv}\mleft(\bigcup_{\by\in(\be_1^{\T}\bx_k,0)^{\T}+\bbB^2/8}\partial_{\dC}^2 f_2(\by)(\bw_0)\mright)}\langle\bz,\bw_0\rangle\\
        &\le\max_{\bz\in\operatorname{conv}\mleft(\bigcup_{\by\in\R\times(-1/4,1/4)}\partial_{\dC}^2 f_2(\by)(\bw_0)\mright)}\langle\bz,\bw_0\rangle\\
        &\le -0.0092714\quad\text{for all $k=0,\ldots,m$},
    \end{aligned}
    $$
    where the last inequality 
    is due to
    the band refutation of $f_2$ in Corollary~\ref{maintext-cor:all-desired}. The proof is complete.
\end{myproof}

\section{Concluding remarks}\label{sec:conclusion}
In this paper, we have shown the intractability of 
deterministic second-order $C^{1,1}$-optimization
(and thus weakly convex optimization).
To be specific, we have shown that
for any $m\ge 1$, any $n\ge 2$ (resp., $n\ge m+3$), and any $\mathbfcal{A}\in\bbA_{\ddetzr}^n(m)$ (resp., $\mathbfcal{A}\in\bbA_{\ddet}^n(m)$), there exists some 
$f\in\bbF_{\!C^1}^n(4.086,17/256;\mathbfcal{A})$
such that the 
iterates $\bx_0,\ldots,\bx_m$ generated by $\mathbfcal{A}$ when applied to $f$ are not $(\varepsilon,\eta,\delta)$-Goldstein approximate second-order stationary for all $\varepsilon\ge 0$, $\eta\in [0,0.0092714)$, and $\delta\in[0,1/8]$.
One natural and interesting future direction is to understand whether we can use randomization to bypass the lower bounds; see, e.g.,~\citep{zhang2020complexity,tian2022finite,davis2022gradient,lin2022gradient,metel2022perturbed,kornowski2024algorithm} for first-order counterparts.


\bibliographystyle{plainnat}
\bibliography{references}

\clearpage
\appendix

\section{Full proof of Theorem~\ref{thm:lb-1}}\label{app:proofs-main}

\subsection{Preparation: The local surgery}\label{sec:local-surgery}
Recall from Section~\ref{sec:proof-theorem} that we shall carry out certain local 
surgeries on an ambient function that is locally identical to $-2^{-1}\mleft\|\bullet\mright\|^2$.
In this subsection, we describe the local surgery in full detail.


To begin, consider the one-dimensional $C^3$ bump function $p_1(\bullet;\delta):\R\rightarrow\R$ defined by
$$
    p_1(x;\delta):=
    \begin{dcases}
        \left(\frac{x^2}{\delta ^2}-1\right)^4, & |x|\le\delta,\\
        0, & \textnormal{otherwise},
    \end{dcases}
$$
where $\delta>0$, 
and its induced
two-dimensional bump function $p_2(\bullet,\bullet;\delta):\R^2\rightarrow\R$ defined by\footnote{It is worth noting that the exponent $4$ in (\ref{eq:def-p2}) is necessary and essential, and its purpose is to enforce a rapid decay 
within a small vicinity near $\bd{0}$, ensuring that certain undesired magnitudes 
that arise in later constructions remain controllable when gated by $p_2(\bullet,\bullet;\delta)$; if we replace $4$ with, e.g., $3$ or $2$, then the validity of the remaining analysis would break down.} 
\begin{equation}\label{eq:def-p2}
    p_2(x,y;\delta):=p_1\bigl(\sqrt{x^2+y^2};\delta\bigr)=
    \begin{dcases}
        \left(\frac{x^2+y^2}{\delta ^2}-1\right)^4, & \sqrt{x^2+y^2}\le \delta,\\
        0, & \textnormal{otherwise},
    \end{dcases}
\end{equation}
which 
inherits the $C^3$-smoothness of $p_1(\bullet;\delta)$:
In light of the composition 
$$
    p_2(\bullet,\bullet;\delta)=p_1(\bullet;\delta)\circ\sqrt{\bullet}\circ\bigl((x,y)\mapsto x^2+y^2\bigr)
$$
and the fact that $\sqrt{\bullet}$ is $C^{\infty}$ on $(0,\infty)$, we know that $p_2(\bullet,\bullet;\delta)$ is $C^3$ on
$\R^2\setminus\{\bd{0}\}$. On the other hand, we see by (\ref{eq:def-p2}) that there exists an open neighborhood of $\bd{0}$ (i.e., $\delta\operatorname{int}(\bbB^2)$) on which $p_2(\bullet,\bullet;\delta)$ is a polynomial; hence, $p_2(\bullet,\bullet;\delta)$ is $C^{\infty}$ around $\bd{0}$. By putting the above two pieces together, it follows that $p_2(\bullet,\bullet;\delta)$ is (globally) $C^3$. 

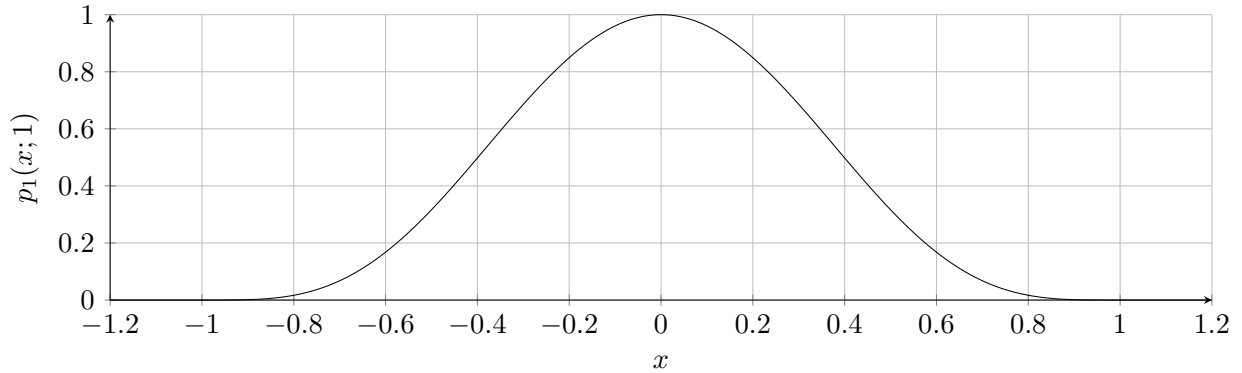
\begin{figure}[!ht]
\resizebox{\linewidth}{!}{%
\centering
\begin{tikzpicture}[
  declare function={
    p1(\x) = (abs(\x) < 1) * (\x^2-1)^4
              + (abs(\x) >= 1) * 0;
  }
]
\begin{axis}[
    axis lines = left,
    grid=major,
    xlabel = $x$,
    ylabel = $p_1(x;1)$,
height=0.33*\linewidth,
width=\linewidth,
]
\addplot [
    domain=-1.2:1.2, 
    samples=500,
    color=black,
]
{p1(x)};
\end{axis}
\end{tikzpicture}
}
\caption{The landscape of $p_1(\bullet;1)$.}
\label{fig:p1-func}
\end{figure}

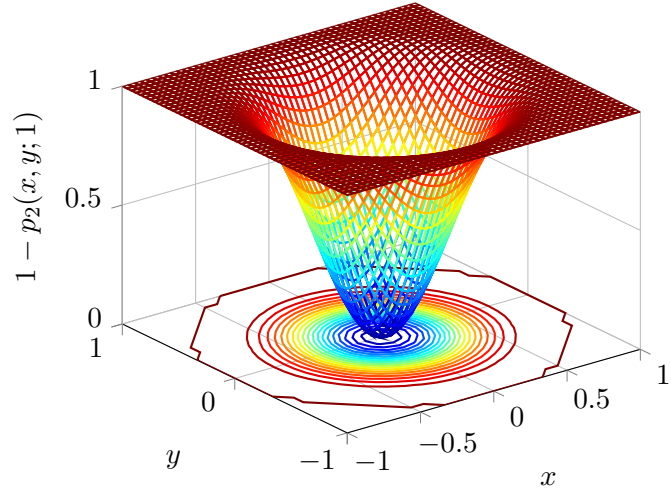
\begin{figure}[!ht]
    \centering
    \begin{tikzpicture}[
  declare function={
    func(\x,\y) = ( \x^2 + \y^2 <= 1 ) ? (\x^2+\y^2-1)^4 : 0;
  }
]
\begin{axis}[
    xlabel={$x$}, ylabel={$y$}, zlabel={$1-p_2(x,y;1)$}, xmin=-1, xmax=1, ymin=-1, ymax=1, zmin=0, zmax=1, 
    grid=both,
    tick align=outside,
    tickpos=left,
    axis lines*=left,
    view={-37.5}{30},
    point meta min=0,  
    point meta max=1,   
    colormap/jet,
]
    \addplot3 [
        contour gnuplot={contour dir=z,
        output point meta=rawz,
        number=20,
        labels=false},
        z filter/.expression={0},
        domain=-1:1,
        domain y=-1:1,
        thick,
    ] {1-func(x,y)};
    
    \addplot3 [mesh,thick,
    samples=50,
    samples y=50,
    domain=-1:1,
    domain y=-1:1,
    ] {1-func(x,y)};
\end{axis}
\end{tikzpicture}
    \caption{The landscape of $1-p_2(\bullet,\bullet;1)$.}
    \label{fig:bump}
\end{figure}


Recall from Section~\ref{sec:proof-theorem} that the local surgery pertains to interpolating between $-2^{-1}\mleft\|\bullet\mright\|^2$ and $(x,y)\mapsto -2^{-1} x^2$. To this end, consider the smooth blend $g(\bullet,\bullet;\delta):\R^2\rightarrow\R$ defined by
$$
\begin{aligned}
    g(x,y;\delta):={}& p_2(x,y;\delta)\cdot\left(-\frac{1}{2} x^2\right)+\bigl(1-p_2(x,y;\delta)\bigr)\cdot\left(-\frac{1}{2} x^2-\frac{1}{2} y^2\right)\\
    ={}&
    \begin{dcases}
        \frac{1}{2} \left(y^2 \left(\frac{x^2+y^2}{\delta ^2}-1\right)^4-x^2-y^2\right), & \sqrt{x^2+y^2}\le \delta,\\
        \frac{1}{2} (-x^2-y^2), & \textnormal{otherwise}.
    \end{dcases}
\end{aligned}
$$

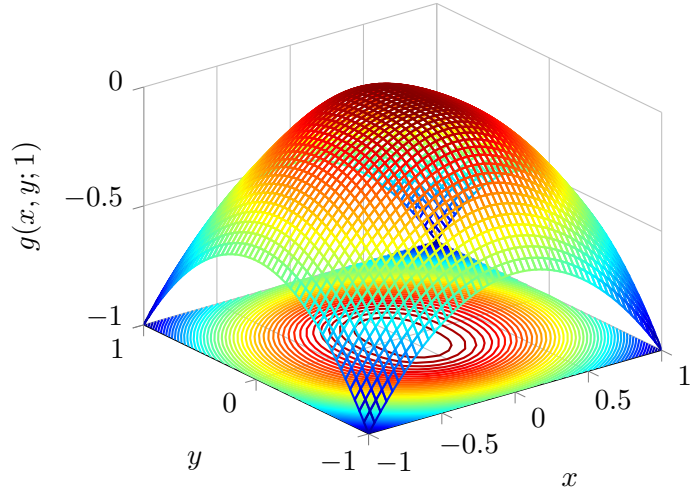
\begin{figure}[!ht]
    \centering
    \begin{tikzpicture}[
  declare function={
    p2(\x,\y) = ( \x^2 + \y^2 <= 1 ) ? (\x^2+\y^2-1)^4 : 0;
    func(\x,\y) = p2(\x,\y)*(-0.5*\x^2)+(1-p2(\x,\y))*(-0.5*\x^2-0.5*\y^2);
  }
]
\begin{axis}[
    xlabel={$x$}, ylabel={$y$}, zlabel={$g(x,y;1)$}, xmin=-1, xmax=1, ymin=-1, ymax=1, zmin=-1, zmax=0, 
    grid=both,
    tick align=outside,
    tickpos=left,
    axis lines*=left,
    view={-37.5}{30},
    point meta min=-1,  
    point meta max=0,   
    colormap/jet,
]
    \addplot3 [
        contour gnuplot={contour dir=z,
        output point meta=rawz,
        number=50,
        labels=false},
        z filter/.expression={-1},
        domain=-1:1,
        domain y=-1:1,
        thick,
    ] {func(x,y)};
    
    \addplot3 [mesh,thick,
    samples=50,
    samples y=50,
    domain=-1:1,
    domain y=-1:1,
    ] {func(x,y)};
    
\end{axis}
\end{tikzpicture}
    \caption{The landscape of $g(\bullet,\bullet;1)$.}
    \label{fig:gadget}
\end{figure}

By applying Leibniz's rule, 
it follows that $g(\bullet,\bullet;\delta)$ is $C^3$ as well, with
$$
    \nabla g(x,y;\delta)=
    \begin{dcases}
        \begin{pmatrix}
        x \left(\frac{4 y^2 (-\delta ^2+x^2+y^2)^3}{\delta ^8}-1\right) \\
        y \left(\left(\frac{x^2+y^2}{\delta ^2}-1\right)^4+\frac{4 y^2 (-\delta ^2+x^2+y^2)^3}{\delta ^8}-1\right)
    \end{pmatrix}
    , & \sqrt{x^2+y^2}\le \delta,\\
        \begin{pmatrix}
        -x \\
        -y
    \end{pmatrix}, & \textnormal{otherwise},
    \end{dcases}
$$
and
\begin{equation*}\label{eq:Hess-g}
\begin{aligned}
    \nabla^2 g(x,y;\delta)=
    \begin{dcases}
        \bH(x,y;\delta)=
        \begin{pmatrix}
            h_{1 1}(x,y;\delta) & h_{1 2}(x,y;\delta) \\
            h_{2 1}(x,y;\delta) & h_{2 2}(x,y;\delta)
        \end{pmatrix}
        , & \sqrt{x^2+y^2}\le \delta,\\
        -\bI, & \textnormal{otherwise},
    \end{dcases}
\end{aligned}
\end{equation*}
where
$$
    \begin{dcases}
        h_{1 1}(x,y;\delta):=\frac{-\delta ^8+4 y^2 (-\delta ^2+x^2+y^2)^2 (-\delta ^2+7 x^2+y^2)}{\delta ^8},\\
        h_{1 2}(x,y;\delta):=\frac{8 x y (-\delta ^2+x^2+y^2)^2 (-\delta ^2+x^2+4 y^2)}{\delta ^8},\\
        h_{2 1}(x,y;\delta):=h_{1 2}(x,y;\delta),\\
        h_{2 2}(x,y;\delta):=\frac{\left(
        \begin{gathered}
        -\delta ^8+(-\delta ^2+x^2+y^2)^4+16 y^2 (-\delta ^2+x^2+y^2)^3\\
        +4 y^2 (-\delta ^2+x^2+y^2)^2 (-\delta ^2+x^2+7 y^2)
        \end{gathered}
        \right)}{\delta ^8},
    \end{dcases}
$$
which immediately imply the desired oracle compatibility mentioned in Section~\ref{sec:proof-theorem}.

\begin{lemma}\label{lma:oracle}
    Let $\delta>0$ be a constant. We have
    $$
        g(0,0;\delta)=0,\quad\nabla g(0,0;\delta)=\bd{0},\quad\text{and}\quad\nabla^2 g(0,0;\delta)=
        \begin{pmatrix}
            -1 & 0 \\
            0 & 0
        \end{pmatrix}.
    $$
\end{lemma}

With Lemma~\ref{lma:oracle} in place, as discussed in Section~\ref{sec:proof-theorem}, it then remains to verify the (stationarity) refutation property of $g(\bullet,\bullet;\delta)$ and to estimate the Lipschitz modulus of $\nabla g(\bullet,\bullet;\delta)$. To this end, a few interesting technical preparations are in order.

\subsubsection{Technical bounds}
In this part, we shall bound the elements of $\bH(x,y;\delta)$ on $\delta\bbB^2$. In view of the fact that 
$$
    h_{i j}(x,y;\delta)=h_{i j}(\delta^{-1} x,\delta^{-1} y;1)\quad\text{for all $i,j=1,2$}
$$
and the following trivial lemma (with $\bA:=\delta^{-1}\bI$), it suffices to assume that $\delta=1$ throughout the proofs that follow.

\begin{lemma}
    Let $f:\R^n\rightarrow\R$ and $\bG:\R^n\rightarrow\R^m$ be arbitrary, $\bA\in\R^{n\times n}$ be invertible. We have
    $$
        \bigl\{f(\bx):\bG(\bx)\le \bd{0}\bigr\}=\bigl\{f(\bA\bx):\bG(\bA\bx)\le \bd{0}\bigr\}.
    $$
\end{lemma}

To put it simple, linear isomorphisms preserve (constrained) function ranges.

\begin{proof}
    Suppose that $v\in\bigl\{f(\bx):\bG(\bx)\le \bd{0}\bigr\}$. Then, we have $v=f(\bx_0)$ for some $\bx_0\in\R^n$ with $\bG(\bx_0)\le \bd{0}$. However, we also have $f(\bx_0)=f(\bA\bx_0^{\prime})$ and $\bG(\bx_0)=\bG(\bA\bx_0^{\prime})$ for $\bx_0^{\prime}:=\bA^{-1}\bx_0$. Hence, it follows that $v=f(\bA\bx_0^{\prime})$ for some $\bx_0^{\prime}\in\R^n$ such that $\bG(\bA\bx_0^{\prime})\le \bd{0}$; i.e., $v\in\bigl\{f(\bA\bx):\bG(\bA\bx)\le \bd{0}\bigr\}$, as desired. The other direction is trivial.
\end{proof}

As an aside, we note for all $x,y\in\R$ with $\sqrt{x^2+y^2}=\delta$ that
$$
    h_{1 1}(x,y;\delta)=-1,\quad h_{1 2}(x,y;\delta)=h_{2 1}(x,y;\delta)=0,\quad\text{and}\quad h_{2 2}(x,y;\delta)=-1;
$$
hence, we shall also assume that $\sqrt{x^2+y^2}\neq\delta$ in the sequel.

\begin{proposition}\label{prop:tech-bound1}
    Let $\delta>0$ be a constant. We have 
    $$
        -\frac{91}{64}\le h_{1 1}(x,y;\delta)\le-\frac{71}{98}\quad\text{for all $x,y\in\R$ such that $\sqrt{x^2+y^2}\le \delta$}.
    $$
\end{proposition}

\begin{figure}[!ht]
    \centering
    \begin{tikzpicture}[
  declare function={
    func(\x,\y) = ( \x^2 + \y^2 <= 1 ) ? 4* \y^2 *(\x^2+\y^2-1)^2 *(7* \x^2+\y^2-1)-1 : -1;
  }
]
\begin{axis}[
    xlabel={$x$}, ylabel={$y$}, zlabel={$h_{1 1}(x,y;1)$}, xmin=-1, xmax=1, ymin=-1, ymax=1, zmin=-1.43, zmax=-0.72, 
    grid=both,
    tick align=outside,
    tickpos=left,
    axis lines*=left,
    view={-37.5}{30},
    point meta min=-1.43,  
    point meta max=-0.72,   
    colormap/jet,
]
    \addplot3 [
        contour gnuplot={contour dir=z,
        output point meta=rawz,
        number=20,
        labels=false},
        z filter/.expression={-1.43},
        domain=-1:1,
        domain y=-1:1,
        thick,
    ] {func(x,y)};
    
    \addplot3 [mesh,thick,
    samples=50,
    samples y=50,
    domain=-1:1,
    domain y=-1:1,
    ] {func(x,y)};
\end{axis}
\end{tikzpicture}
    \caption{The landscape of $h_{1 1}(\bullet,\bullet;1)$; for visual clarity, we overwrite $h_{1 1}(x,y;1)=-1$ on $(\bbB^2)^{\complement}$.}
    \label{fig:h11}
\end{figure}
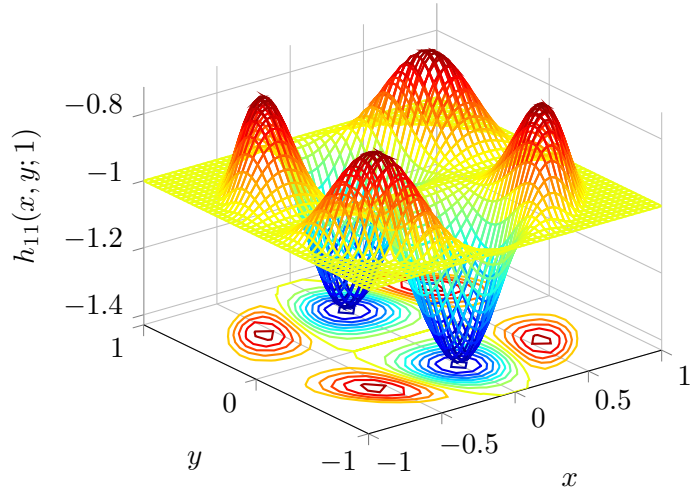

\begin{proof}
    To begin, we calculate
    $$
        h_{1 1}(x,y;1)=4 y^2 (x^2+y^2-1)^2 (7 x^2+y^2-1)-1.
    $$
    In what follows, we shall derive a set of candidate stationary points of $h_{1 1}(\bullet,\bullet;1)$ and then evaluate and compare at these candidates; as $h_{1 1}(x,0;1)=-1$ for all $x\in\R$, we shall also assume that $y\neq 0$.
    
    It can be calculated that
    $$
        \frac{\partial}{\partial x} h_{1 1}(x,y;1)=56 x y^2 (x^2+y^2-1)^2+16 x y^2 (x^2+y^2-1) (7 x^2+y^2-1)
    $$
    and
    $$
        \frac{\partial}{\partial y} h_{1 1}(x,y;1)=8 y (x^2+y^2-1)^2 (7 x^2+y^2-1)+8 y^3 (x^2+y^2-1)^2+16 y^3 (x^2+y^2-1) (7 x^2+y^2-1).
    $$
    By setting these partial derivatives to zero and simplifying, we obtain
    \begin{equation*}
        \begin{dcases}
            24 x y^2 (x^2+y^2-1) (7 x^2+3 y^2-3)=0,\\
            8 y (x^2+y^2-1) \bigl(7 x^4+x^2 (23 y^2-8)+4 y^4-5 y^2+1\bigr)=0.
        \end{dcases}
    \end{equation*}
    By using the assumptions that $\sqrt{x^2+y^2}\neq 1$ and $y\neq 0$, we can further simplify the system to
    \begin{equation}\label{eq:system-1}
        \begin{dcases}
            x (7 x^2+3 y^2-3)=0,\\
            7 x^4+x^2 (23 y^2-8)+4 y^4-5 y^2+1=0.
        \end{dcases}
    \end{equation}
    As the second equation is quadratic in $y^2$, we can solve for $y^2$ and this yields
    \begin{equation}\label{eq:sol-y^2}
        y^2=\frac{1}{8} \left(-23 x^2\pm\sqrt{3} \sqrt{139 x^4-34 x^2+3}+5\right).
    \end{equation}
    We divide the subsequent discussions into two cases.
    \begin{itemize}
        \item Suppose that $y^2=8^{-1} \bigl(-23 x^2-\sqrt{3} \sqrt{139 x^4-34 x^2+3}+5\bigr)$. Then, by substituting this into the first equation in (\ref{eq:system-1}) and simplifying, it follows that either $x=0$ or 
    $$
        9\le 13 x^2+3 \sqrt{417 x^4-102 x^2+9}+9=0,
    $$
    which is impossible. Hence, we only have one pair of solutions in this case; i.e., $(0,\pm 2^{-1})$. 
        \item Suppose that $y^2=8^{-1} \bigl(-23 x^2+\sqrt{3} \sqrt{139 x^4-34 x^2+3}+5\bigr)$. Then, we similarly know that either $x=0$ or 
    $$
        3 \sqrt{417 x^4-102 x^2+9}=13 x^2+9 
        \implies x=0,\pm \frac{3}{2 \sqrt{7}}.
    $$
    Hence, in this case, we have the following four solutions
    $$
        \left(\pm \frac{3}{2 \sqrt{7}},\pm\frac{1}{2}\right);
    $$
    the other pair of solutions $(0,\pm 1)$ is discarded as the assumption $\sqrt{x^2+y^2}\neq 1$ is violated.
    \end{itemize}    
    
    By comparing $h_{1 1}(x,y;1)$ at all calculated points and boundary cases 
    (that are 
    explicitly excluded by
    the assumptions),
    the desired claim directly follows. The proof is complete.
\end{proof}

\begin{proposition}\label{prop:tech-bound2}
    Let $\delta>0$ be a constant. 
    For
    all $x,y\in\R$ with $\sqrt{x^2+y^2}\le \delta$, we have 
    $$
        \bigl|h_{1 2}(x,y;\delta)\bigr|\le
        \frac{1}{2} \sqrt{\frac{1}{2} \bigl(\sqrt{177}+17\bigr)} \left(\frac{1}{64} \bigl(\sqrt{177}+17\bigr)-\frac{7}{8}\right)^2 \left(\frac{1}{16} \bigl(\sqrt{177}+17\bigr)-\frac{7}{8}\right)\le 0.32.
    $$
\end{proposition}

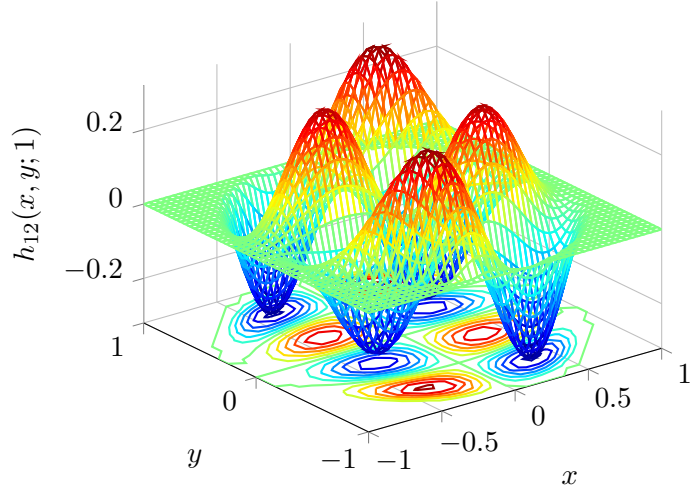
\begin{figure}[!ht]
    \centering
    \begin{tikzpicture}[
  declare function={
    func(\x,\y) = ( \x^2 + \y^2 <= 1 ) ? 8* \x* \y* (\x^2+\y^2-1)^2 *(\x^2+4* \y^2-1) : 0;
  }
]
\begin{axis}[
    xlabel={$x$}, ylabel={$y$}, zlabel={$h_{1 2}(x,y;1)$}, xmin=-1, xmax=1, ymin=-1, ymax=1, zmin=-0.32, zmax=0.32, 
    grid=both,
    tick align=outside,
    tickpos=left,
    axis lines*=left,
    view={-37.5}{30},
    point meta min=-0.32,  
    point meta max=0.32,   
    colormap/jet,
]
    \addplot3 [
        contour gnuplot={contour dir=z,
        output point meta=rawz,
        number=20,
        labels=false},
        z filter/.expression={-0.32},
        domain=-1:1,
        domain y=-1:1,
        thick,
    ] {func(x,y)};
    
    \addplot3 [mesh,thick,
    samples=50,
    samples y=50,
    domain=-1:1,
    domain y=-1:1,
    ] {func(x,y)};
\end{axis}
\end{tikzpicture}
    \caption{The landscape of $h_{1 2}(\bullet,\bullet;1)$; for visual clarity, we overwrite $h_{1 2}(x,y;1)=0$ on $(\bbB^2)^{\complement}$.}
    \label{fig:h12}
\end{figure}

\begin{proof}
    Similar to the proof of Proposition~\ref{prop:tech-bound1}, we first calculate
    $$
        h_{1 2}(x,y;1)=8 x y (x^2+y^2-1)^2 (x^2+4 y^2-1),
    $$
    together with
    $$
        \frac{\partial}{\partial x}h_{1 2}(x,y;1)=16 x^2 y (x^2+y^2-1)^2+32 x^2 y (x^2+y^2-1) (x^2+4 y^2-1)+8 y (x^2+y^2-1)^2 (x^2+4 y^2-1)
    $$
    and 
    $$
        \frac{\partial}{\partial y}h_{1 2}(x,y;1)=64 x y^2 (x^2+y^2-1)^2+32 x y^2 (x^2+y^2-1) (x^2+4 y^2-1)+8 x (x^2+y^2-1)^2 (x^2+4 y^2-1);
    $$
    as $h_{1 2}(x,y;1)=0$ as soon as $x=0$ or $y=0$, we further suppose that $x\neq 0$ and $y\neq 0$ in the sequel. By setting these partial derivatives to zero and simplifying, we obtain
    $$
        \begin{dcases}
            8 y (x^2+y^2-1) \bigl(7 x^4+x^2 (23 y^2-8)+4 y^4-5 y^2+1\bigr)=0,\\
            8 x (x^2+y^2-1) \bigl(17 (x^2-1) y^2+(x^2-1)^2+28 y^4\bigr)=0,
        \end{dcases}
    $$
    which, upon incorporating all the assumptions mentioned earlier, further simplifies to
    $$
        \begin{dcases}
            7 x^4+x^2 (23 y^2-8)+4 y^4-5 y^2+1=0,\\
            17 (x^2-1) y^2+(x^2-1)^2+28 y^4=0.
        \end{dcases}
    $$
    As the first equation is the same to the second one in (\ref{eq:system-1}), we know that (\ref{eq:sol-y^2}) remains valid here. Similarly, we divide the subsequent discussions into two cases.
    \begin{itemize}
        \item Suppose that $y^2=8^{-1} \bigl(-23 x^2-\sqrt{3} \sqrt{139 x^4-34 x^2+3}+5\bigr)$. Then, we know from the second equation that
    $$
        (8 x^2-1) \left(61 x^2+3 \sqrt{417 x^4-102 x^2+9}-7\right)=0;
    $$
    i.e., either $x^2=1/8$ or
    $$
        9 \le 61 x^2+3 \sqrt{417 x^4-102 x^2+9}=7,
    $$
    which is impossible. Hence, the solutions in this case are merely
    $$
        \left(\pm\frac{1}{2\sqrt{2}},\pm\frac{1}{8} \sqrt{17-\sqrt{177}}\right).
    $$
        \item Suppose that $y^2=8^{-1} \bigl(-23 x^2+\sqrt{3} \sqrt{139 x^4-34 x^2+3}+5\bigr)$. Then, we similarly have either $x^2=1/8$ or
    $$
        -61 x^2+3 \sqrt{417 x^4-102 x^2+9}+7=0 
        \implies x=\pm 1,
    $$
    which should be discarded as the assumption $\sqrt{x^2+y^2}\neq 1$ is violated. Hence, only the below solution pairs arise in this case
    $$
        \left(\pm\frac{1}{2\sqrt{2}},\pm\frac{1}{8}\sqrt{17+\sqrt{177}}\right).
    $$
    \end{itemize}
    The desired claim then directly follows by putting the above pieces together. We are done.
\end{proof}

\begin{proposition}\label{prop:tech-bound3}
    Let $\delta>0$ be a constant. We have 
    $$
        -\frac{253}{125}\le h_{2 2}(x,y;\delta)\le0\quad\text{for all $x,y\in\R$ such that $\sqrt{x^2+y^2}\le \delta$}.
    $$
\end{proposition}

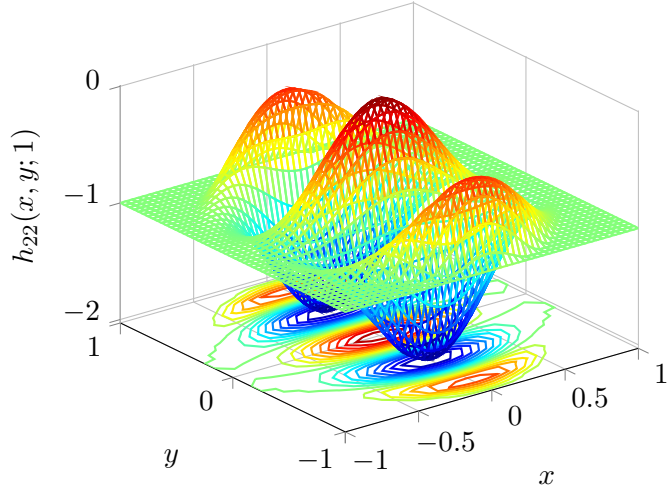
\begin{figure}[!ht]
    \centering
    \begin{tikzpicture}[
  declare function={
    func(\x,\y) = ( \x^2 + \y^2 <= 1 ) ? (\x^2+\y^2-1)^4+16 *\y^2 *(\x^2+\y^2-1)^3+4 *\y^2 *(\x^2+7* \y^2-1)* (\x^2+\y^2-1)^2-1 : -1;
  }
]
\begin{axis}[
    xlabel={$x$}, ylabel={$y$}, zlabel={$h_{2 2}(x,y;1)$}, xmin=-1, xmax=1, ymin=-1, ymax=1, zmin=-2.03, zmax=0.01, 
    grid=both,
    tick align=outside,
    tickpos=left,
    axis lines*=left,
    view={-37.5}{30},
    point meta min=-2.03,  
    point meta max=0.01,   
    colormap/jet,
]
    \addplot3 [
        contour gnuplot={contour dir=z,
        output point meta=rawz,
        number=20,
        labels=false},
        z filter/.expression={-2.03},
        domain=-1:1,
        domain y=-1:1,
        thick,
    ] {func(x,y)};
    
    \addplot3 [mesh,thick,
    samples=50,
    samples y=50,
    domain=-1:1,
    domain y=-1:1,
    ] {func(x,y)};
\end{axis}
\end{tikzpicture}
    \caption{The landscape of $h_{2 2}(\bullet,\bullet;1)$; for visual clarity, we overwrite $h_{2 2}(x,y;1)=-1$ on $(\bbB^2)^{\complement}$.}
    \label{fig:h22}
\end{figure}

\begin{proof}
    In a similar vein, we begin by calculating
    $$
        h_{2 2}(x,y;1)=(x^2+y^2-1)^4+16 y^2 (x^2+y^2-1)^3+4 y^2 (x^2+7 y^2-1) (x^2+y^2-1)^2-1,
    $$
    together with
    $$
        \frac{\partial}{\partial x}h_{2 2}(x,y;1)=8 x (x^2+y^2-1)^3+104 x y^2 (x^2+y^2-1)^2+16 x y^2 (x^2+7 y^2-1) (x^2+y^2-1)
    $$
    and
    $$
    \begin{aligned}
        \frac{\partial}{\partial y}h_{2 2}(x,y;1)={}& 40 y (x^2+y^2-1)^3+8 y (x^2+y^2-1)^2 (x^2+7 y^2-1)\\
        {}&+152 y^3 (x^2+y^2-1)^2+16 y^3 (x^2+y^2-1) (x^2+7 y^2-1).
    \end{aligned}
    $$
    By setting these partial derivatives to zero and simplifying, we obtain
    $$
        \begin{dcases}
            8 x (x^2+y^2-1) \bigl(17 (x^2-1) y^2+(x^2-1)^2+28 y^4\bigr)=0,\\
            24 y (x^2+y^2-1) (2 x^2+3 y^2-2) (x^2+5 y^2-1)=0,
        \end{dcases}
    $$
    which can be further simplified to
    $$
        \begin{dcases}
            x \bigl(17 (x^2-1) y^2+(x^2-1)^2+28 y^4\bigr)=0,\\
            y (2 x^2+3 y^2-2) (x^2+5 y^2-1)=0.
        \end{dcases}
    $$
    As the second equation is clearly more amenable to analysis, we shall exploit it first. This time, the subsequent discussions are divided into three cases. 
    \begin{itemize}
        \item Suppose that $y=0$. Then, we know from the other equation that
    $$
        x (x^2-1)^2=0 
        \implies
        x=0,\pm 1;
    $$
    i.e., two pairs of solutions, namely $(0,0)$ and $(\pm 1,0)$, arise in this case, but the latter should be discarded as it violates the assumption $\sqrt{x^2+y^2}\neq 1$.

        \item Suppose that $y^2=2(1-x^2)/3$. Then, we have
    $$
        x^3=x 
        \implies
        x=0,\pm 1;
    $$
    hence, $(0,\pm\sqrt{2/3})$ are the only solutions in this case, as $(\pm 1,0)$ should be discarded as well. 
    
        \item Suppose that $y^2=(1-x^2)/5$. Then, it follows, again, that
    $$
        x^3=x 
        \implies
        x=0,\pm 1;
    $$
    likewise, the solutions in this case are merely $(0,\pm\sqrt{1/5})$.
    \end{itemize}
    Summarizing the above discussions yields the desired result and thus completes the proof.
\end{proof}

\subsubsection{Stationarity refutations}
With Propositions~\ref{prop:tech-bound1},~\ref{prop:tech-bound2}, and~\ref{prop:tech-bound3} in place, we are ready to control $\langle\nabla^2 g(x,y;\delta),\bw\bw^{\T}\rangle$ along some specific $\bw\in\bbS^2$ that will prove useful in later developments.

\begin{corollary}\label{cor:uniformity}
    Let $\delta>0$ be a constant
    and
    $\bw:=(10,11)^{\T}/\sqrt{221}\in\bbS^{2}$. We have 
    \begin{equation}\label{eq:tech-est1}
        \begin{dcases}
            \langle\bH(x,y;\delta),\bw\bw^{\T}\rangle\le-0.0092714\quad\text{for all $x,y\in\R$ such that $\sqrt{x^2+y^2} \le \delta$}, \\
            \langle-\bI,\bw\bw^{\T}\rangle=-1\le-0.0092714\quad\text{for all $x,y\in\R$ such that $\sqrt{x^2+y^2} \ge \delta$}.
        \end{dcases}
    \end{equation}
    As a consequence, we have $\langle\nabla^2 g(x,y;\delta),\bw\bw^{\T}\rangle\le-0.0092714$ for all $x,y\in\R$.
\end{corollary}

We remark that the said consequence is immediate from the $C^3$-smoothness of $g(\bullet,\bullet;\delta)$. 

\begin{proof}
    As the claim $\langle-\bI,\bw\bw^{\T}\rangle=-1$ is trivial, we thus focus on the remaining claim in (\ref{eq:tech-est1}). Indeed, it follows from Propositions~\ref{prop:tech-bound1},~\ref{prop:tech-bound2}, and~\ref{prop:tech-bound3} that for all $x,y\in\R$ with $\sqrt{x^2+y^2} \le \delta$,
    $$
    \begin{aligned}
        \langle\bH(x,y;\delta),\bw\bw^{\T}\rangle&=\frac{1}{221}\bigl(100 h_{1 1}(x,y;\delta)+220 h_{1 2}(x,y;\delta)+121 h_{2 2}(x,y;\delta)\bigr)\\
        &\le\frac{1}{221}\left(100 \left(-\frac{71}{98}\right) +220 \cdot 0.32\right)\le -0.0092714,
    \end{aligned}
    $$
    as desired. The proof is complete.
\end{proof}

\subsubsection{Lipschitz modulus estimations}
We next turn to estimate $\operatorname{lip}\nabla g(\bullet,\bullet;\delta)$.

\begin{corollary}\label{cor:spectral-bound2}
    Let $\delta>0$ be a constant. We have 
    $$
        \begin{dcases}
            \|\bH(x,y;\delta)\|_{\sigma}\le 4.086\quad\text{for all $x,y\in\R$ such that $\sqrt{x^2+y^2} \le \delta$},\\
            \mleft\|-\bI\mright\|_{\sigma}=1\le 4.086\quad\text{for all $x,y\in\R$ such that $\sqrt{x^2+y^2} \ge \delta$}.
        \end{dcases}
    $$
    As a consequence, we have $\operatorname{lip}\nabla g(x,y;\delta)\le 4.086$ for all $x,y\in\R$.
\end{corollary}

We note that the said consequence is immediate from~\citep[Theorem~9.7]{rockafellar2009variational}.

\begin{proof}
    Recall that $\|\bX\|_{\sigma}\le\|\bX\|_1$ for every matrix $\bX$; see, e.g.,~\citep[Proposition~4.2]{chen2020tensor}. Hence, it follows from Propositions~\ref{prop:tech-bound1},~\ref{prop:tech-bound2}, and~\ref{prop:tech-bound3} that for all $x,y\in\R$ with $\sqrt{x^2+y^2} \le \delta$,
    $$
    \begin{aligned}
        \|\bH(x,y;\delta)\|_{\sigma}\le\|\bH(x,y;\delta)\|_1&=\bigl|h_{1 1}(x,y;\delta)\bigr| + 2\,\bigl| h_{1 2}(x,y;\delta)\bigr| + \bigl|h_{2 2}(x,y;\delta)\bigr|\\
        &\le \frac{91}{64}+2\cdot 0.32+ \frac{253}{125}=4.085875\le 4.086,
    \end{aligned}
    $$
    as desired. The proof is complete as the other claim is trivial.
\end{proof}

\subsection{Preparation: The ambient function}
An essential object that has appeared in preceding discussions but has not yet been 
made precise
is the ambient function that is locally identical to $-2^{-1}\mleft\|\bullet\mright\|^2$.
In fact, this function is constructed by horizontally gluing shifted copies of the (cross-supported) local ambient function introduced below, with one copy placed at each iterate; recall from Section~\ref{sec:proof-theorem} that all the iterates lie on the $x$-axis.

To begin, consider the function $q(\bullet;\gamma):\R\rightarrow\R$ defined by
\begin{equation}\label{eq:def-q}
    q(x;\gamma):=
    \begin{dcases}
        -\frac{1}{2}x^2, & |x| \le \frac{\gamma}{4},\\
        \frac{1}{16} (-8 \gamma | x| +\gamma^2+8 x^2), & \frac{\gamma}{4}\le |x| \le \frac{3 \gamma}{4},\\
        -\frac{1}{2} (\gamma-|x| )^2, & \frac{3 \gamma}{4}\le |x| \le \gamma,\\
        0, & \textnormal{otherwise},
    \end{dcases}
\end{equation}
where $\gamma>0$.
This function is $C^{1,1}$, and it can be directly calculated that
\begin{equation}\label{eq:hess-correction}
    q^{\prime}(x;\gamma)=
    \begin{dcases}
        -x, & |x| \le \frac{\gamma}{4},\\
        x-\frac{\gamma}{2}\operatorname{sign}(x), & \frac{\gamma}{4}\le |x| \le \frac{3 \gamma}{4},\\
        \gamma \operatorname{sign}(x)-x, & \frac{3 \gamma}{4}\le |x| \le \gamma,\\
        0, & \textnormal{otherwise},
    \end{dcases}
    \quad\text{and}\quad
    q^{\prime\prime}(x;\gamma)=
    \begin{dcases}
        -1, & |x| < \frac{\gamma}{4},\\
        1, & \frac{\gamma}{4}< |x| < \frac{3 \gamma}{4},\\
        -1, & \frac{3 \gamma}{4}< |x| < \gamma,\\
        0, & |x| > \gamma.
    \end{dcases}
\end{equation}

\begin{figure}[!ht]
\resizebox{\linewidth}{!}{%
\centering
\begin{tikzpicture}[
  declare function={
    q(\x) = (abs(\x) < 0.25) * (-0.5*\x^2)
              + (abs(\x) >= 0.25 && abs(\x) < 0.75) * ((-8*abs(\x)+1+8*\x^2)/16)
              + (abs(\x) >= 0.75 && abs(\x) < 1) * (-0.5*(1-abs(\x))^2)
              + (abs(\x) >= 1) * 0;
  }
]
\begin{axis}[
    axis lines = left,
    grid=major,
    xlabel = $x$,
    ylabel = $q(x;1)$,
height=0.33*\linewidth,
width=\linewidth,
]
\addplot [
    domain=-1.2:1.2, 
    samples=500,
    color=black,
]
{q(x)};
\end{axis}
\end{tikzpicture}}
\caption{The landscape of $q(\bullet;1)$.}
\label{fig:q-func}
\end{figure}
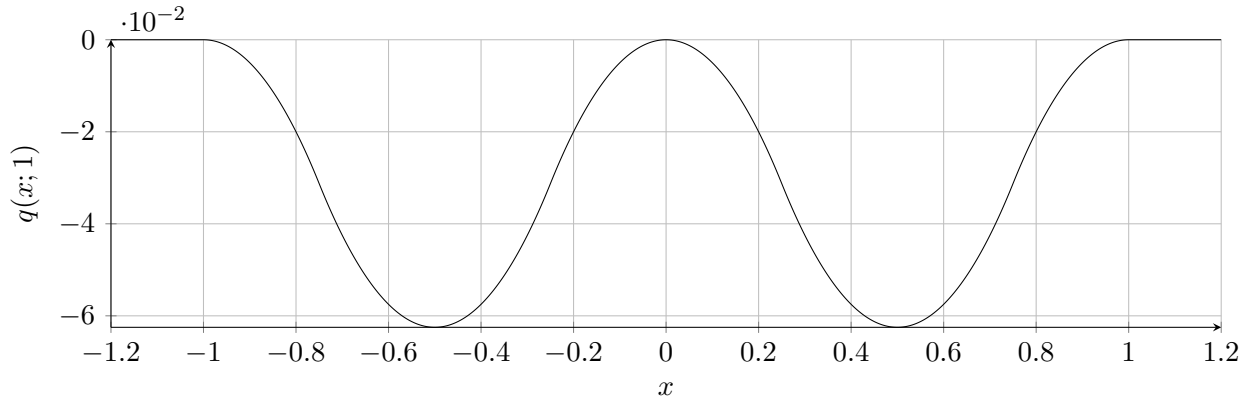

With $q(\bullet;\gamma)$ in place, we are ready to introduce the local ambient function $h(\bullet,\bullet;\gamma):\R^2\rightarrow\R$ defined by
\begin{equation}\label{eq:def-h}
    h(x,y;\gamma):=q(x;\gamma)+q(y;1),
\end{equation}
which is (trivially) $C^{1,1}$ as well. We refer to this function as ``ambient'' as it satisfies that
\begin{equation}\label{eq:h-equals}
    h(x,y;\gamma)=-\frac{1}{2} x^2-\frac{1}{2} y^2=-2^{-1}\mleft\|\bullet\mright\|^2 (x,y)\quad\text{for all $(x,y)^{\T}\in\left[-\frac{\gamma}{4},\frac{\gamma}{4}\right]\times\left[-\frac{1}{4},\frac{1}{4}\right]$}.
\end{equation}
As in Section~\ref{sec:local-surgery}, we shall proceed with the refutation and smoothness properties of $h(\bullet,\bullet;\gamma)$.

\begin{figure}[!ht]
    \centering
    \begin{tikzpicture}[
  declare function={
    q(\x) = (abs(\x) < 0.25) * (-0.5*\x^2)
              + (abs(\x) >= 0.25 && abs(\x) < 0.75) * ((-8*abs(\x)+1+8*\x^2)/16)
              + (abs(\x) >= 0.75 && abs(\x) < 1) * (-0.5*(1-abs(\x))^2)
              + (abs(\x) >= 1) * 0;
    func(\x,\y) = q(\x)+q(\y);
  }
]
\begin{axis}[
    zticklabel style={/pgf/number format/fixed, /pgf/number format/precision=2},
    xlabel={$x$}, ylabel={$y$}, zlabel={$h(x,y;1)$}, xmin=-1.25, xmax=1.25, ymin=-1.25, ymax=1.25, zmin=-0.125, zmax=0, 
    grid=both,
    tick align=outside,
    tickpos=left,
    axis lines*=left,
    view={-37.5}{30},
    point meta min=-0.125,  
    point meta max=0,   
    colormap/jet,
]
    \addplot3 [
        contour gnuplot={contour dir=z,
        output point meta=rawz,
        number=20,
        labels=false},
        z filter/.expression={-0.125},
        domain=-1.25:1.25,
        domain y=-1.25:1.25,
        thick,
    ] {func(x,y)};
    
    \addplot3 [mesh,thick,
    samples=50,
    samples y=50,
    domain=-1.25:1.25,
    domain y=-1.25:1.25,
    ] {func(x,y)};
    
\end{axis}
\end{tikzpicture}
    \caption{The landscape of $h(\bullet,\bullet;1)$.}
    \label{fig:ambient-local}
\end{figure}
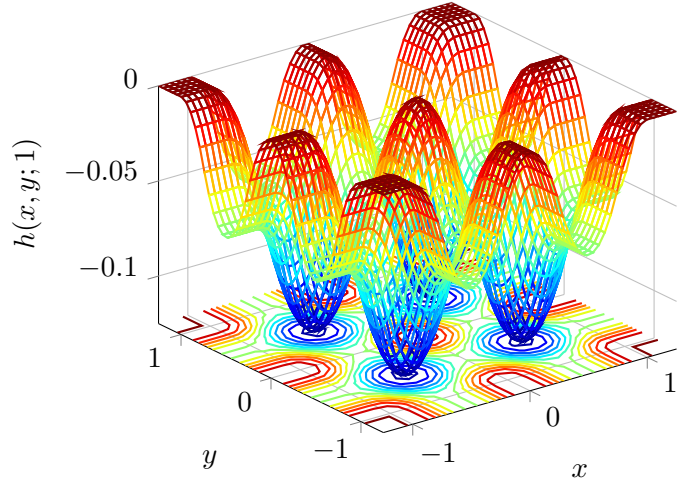

\subsubsection{Stationarity refutations}
By construction, the function $h(\bullet,\bullet;\gamma)$ enjoys a similar stationarity refutation to Corollary~\ref{cor:uniformity}.

\begin{proposition}\label{prop:ambient-refutation}
    Let $\gamma>0$ be a constant
    and
    $\bw:=(10,11)^{\T}/\sqrt{221}\in\bbS^{2}$. We have 
    $$
        \max\bigl\{\langle\bz,\bw\rangle:\bz\in\partial_{\dC}^2 h(x,y;\gamma)(\bw)\bigr\}\le-0.095\quad\text{for all $(x,y)^{\T}\in\R\times\left(-\frac{1}{4},\frac{1}{4}\right)$}.
    $$
\end{proposition}

\begin{proof}
    As $h(\bullet,\bullet;\gamma)$ carries a separable structure, 
    for all $x\in\R$ and $y\in(-4^{-1},4^{-1})$, we have 
    $$
    \begin{aligned}
        \limsup_{\substack{(x^{\prime},y^{\prime})^{\T}\rightarrow(x,y)^{\T}\\(x^{\prime},y^{\prime})^{\T}\in\bbD\times(-1/4,1/4)}}\mleft\{\nabla^2 h(x^{\prime},y^{\prime};\gamma)\mright\}&=\limsup_{\substack{(x^{\prime},y^{\prime})^{\T}\rightarrow(x,y)^{\T}\\(x^{\prime},y^{\prime})^{\T}\in\bbD\times(-1/4,1/4)}}\mleft\{\begin{pmatrix}
            q^{\prime\prime}(x^{\prime};\gamma) & 0 \\
            0 & -1
        \end{pmatrix}\mright\}\\
        &=\begin{pmatrix}
            \displaystyle\limsup_{\substack{x^{\prime}\rightarrow x\\x^{\prime}\in\bbD}}\bigl\{q^{\prime\prime}(x^{\prime};\gamma)\bigr\} & 0 \\
            0 & -1
        \end{pmatrix},
    \end{aligned}    
    $$ 
    where $\mathbb{D}\subseteq\R$ is the set of points at which $q(\bullet;\gamma)$ is twice differentiable. By taking convex hulls on both sides, it further follows that
    $$
        \partial_{\dC}^2 h(x,y;\gamma)=
        \begin{pmatrix}
            \partial_{\dC}^2 q(x;\gamma) & 0 \\
            0 & -1
        \end{pmatrix}
        \quad\text{for all $(x,y)^{\T}\in\R\times\left(-\frac{1}{4},\frac{1}{4}\right)$}.
    $$
    However, we know from (\ref{eq:hess-correction}) that $\partial_{\dC}^2 q(x;\gamma)\subseteq[-1,1]$ for all $x\in\R$. As a result, it follows that
    $$
    \begin{aligned}
        \max\bigl\{\langle\bz,\bw\rangle:\bz\in\partial_{\dC}^2 h(x,y;\gamma)(\bw)\bigr\}
        &=\max\mleft\{\frac{1}{221} (100 a -121):a\in\partial_{\dC}^2 q(x;\gamma)\mright\}\\
        &\le\max\mleft\{\frac{1}{221} (100 a -121):a\in[-1,1]\mright\}\\
        &=-\frac{21}{221}\le-0.095\quad\text{for all $(x,y)^{\T}\in\R\times\left(-\frac{1}{4},\frac{1}{4}\right)$},
    \end{aligned}
    $$
    as desired. The proof is complete.
\end{proof}

\subsubsection{Lipschitz modulus estimations}
We next take some time to understand $\operatorname{lip} \nabla h(\bullet,\bullet;\gamma)$.
\begin{lemma}\label{lma:lip-h}
    Let $\gamma>0$ be a constant. We have $\operatorname{lip}\nabla h(x,y;\gamma)\le 1$ for all $x,y\in\R$.
\end{lemma}

\begin{proof}
    As $q^{\prime}(\bullet;\gamma)$ is $1$-Lipschitz continuous by~\citep[Theorem~9.13]{rockafellar2009variational}, it immediately follows that
$$
\begin{aligned}
    \|\nabla h(x_1,y_1;\gamma)-\nabla h(x_2,y_2;\gamma)\|&=\left\|\begin{pmatrix}
        q^{\prime}(x_1;\gamma) \\
        q^{\prime}(y_1;1)
    \end{pmatrix}-
    \begin{pmatrix}
        q^{\prime}(x_2;\gamma) \\
        q^{\prime}(y_2;1)
    \end{pmatrix}\right\|\\
    &=\sqrt{\bigl(q^{\prime}(x_1;\gamma)-q^{\prime}(x_2;\gamma)\bigr)^2+\bigl(q^{\prime}(y_1;1)-q^{\prime}(y_2;1)\bigr)^2}\\
    &\le\sqrt{(x_1-x_2)^2+(y_1-y_2)^2}=\left\|\begin{pmatrix}
        x_1 \\
        y_1
    \end{pmatrix}-
    \begin{pmatrix}
        x_2 \\
        y_2
    \end{pmatrix}\right\|;
\end{aligned}
$$
i.e., $\nabla h(\bullet,\bullet;\gamma)$ is $1$-Lipschitz continuous as well, from which the desired claim directly follows. 
\end{proof}

\subsection{Proof of Theorem~\ref{thm:lb-1}}\label{app:proof-theorem}
In this subsection, we 
prove Theorem~\ref{thm:lb-1} in a discussion manner.
In what follows, we shall restrict attention to the case $n=2$; for $n\ge 3$, it suffices to extend the construction to be given for $n=2$ in the canonical way.
Let $m\ge 1$ be arbitrary.

To begin, consider the resisting oracle $\mathbfcal{O}:\R^2\rightarrow\R\times\R^2\times\R_{\operatorname{sym}}^{2\times 2}$ defined by
\begin{equation*}
    \mathbfcal{O}(\bx):=\left(0,\bd{0},\begin{pmatrix}
            -1 & 0 \\
            0 & 0
        \end{pmatrix}\right)\quad\text{for all $\bx\in\R^2$}
\end{equation*}
and an arbitrary $\mathbfcal{A}\in\bbA_{\ddetzr}^2(m)$.
Let $\bx_0,\ldots,\bx_m$ be the iterates generated by $\mathbfcal{A}$ after $m$ rounds of interaction with $\mathbfcal{O}$. 
By 
the zero-respecting property of $\mathbfcal{A}$,
we have
$\bx_0=\bd{0}$ 
and 
$\operatorname{supp}(\bx_k)\subseteq \{1\}$ for all $k=0,\ldots,m$; i.e., $\bx_0,\ldots,\bx_m$ lie on the $x$-axis.
Without loss of generality, we may assume that $\bx_i\neq \bx_j$ for all $i\neq j$, and, by relabeling (and abusing notation), that $\bx_0,\ldots,\bx_m$ are sorted in ascending order of their $x$-coordinates.
In particular, this implies that $\bx_0$ may no longer be $\bd{0}$ anymore.
For ease of exposition, we also define $\bx_{-1}:=(-\infty,0)^{\T}$ and $\bx_{m+1}=(\infty,0)^{\T}$, and with a further 
abuse of notation, we denote by $x_k$ the $x$-coordinate of $\bx_k$ for $k=-1,\ldots,m+1$; e.g., $x_{-1}=-\infty$.
In what follows, as discussed in Section~\ref{sec:proof-theorem}, we shall construct a hard function that is compatible with $\mathbfcal{O}$ along $\bx_0,\ldots,\bx_m$ and possesses all the desired properties stated in Theorem~\ref{thm:lb-1}.

Define
$$
    \zeta:=\min\{\xi,1\}\in(0,1],\quad\text{where $\xi:=\min\bigl\{|x_i-x_j|:i,j=0,\ldots,m~\text{with}~i\neq j\bigr\}>0$}.
$$
Consider the hard function $f(\bullet,\bullet;\zeta):\R^2\rightarrow\R$ given by
\begin{equation}\label{eq:def-f}
    f(x,y;\zeta):=
    \begin{dcases}
        f_k(x,y;\zeta), & |x-x_k|\le \zeta/4,\\
        q(y;1), & \textnormal{otherwise},
    \end{dcases}
\end{equation}
where for $k=0,\ldots,m$,
$$
    f_k(x,y;\zeta):=p_2(x-x_k,y;\zeta/32)\cdot\left(-\frac{1}{2} (x-x_k)^2\right)+\bigl(1-p_2(x-x_k,y;\zeta/32)\bigr)\cdot h(x-x_k,y;\zeta/4).
$$
We begin the analysis of $f(\bullet,\bullet;\zeta)$ with a basic observation on the piecewise structure of $f_k(\bullet,\bullet;\zeta)$,
which makes clear that it results from the local surgery performed on the local ambient function.
\begin{lemma}
\label{lma:trivial-obs}
    For
    all $k=0,\ldots,m$, we have 
    \begin{equation}\label{eq:fact-fk}
        f_k(x,y;\zeta)=
        \begin{dcases}
            h(x-x_k,y;\zeta/4), & (x,y)^{\T}\notin (x_k,0)^{\T}+\frac{\zeta}{32}\bbB^2,\\
            g(x-x_k,y;\zeta/32), & (x,y)^{\T}\in\left[x_k-\frac{\zeta/4}{4},x_k+\frac{\zeta/4}{4}\right]\times\left[-\frac{1}{4},\frac{1}{4}\right].
        \end{dcases}
    \end{equation}
\end{lemma}

Before presenting the proof, for
convenience, let us introduce for all $k=0,\ldots,m$ the shorthand
\begin{equation}\label{eq:Bk}
    \bbB_k:=(x_k,0)^{\T}+\frac{\zeta}{32}\bbB^2\subseteq(x_k,0)^{\T}+\frac{1}{2}\cdot\left(\left[-\frac{\zeta/4}{4},\frac{\zeta/4}{4}\right]\times\left[-\frac{1}{4},\frac{1}{4}\right]\right),
\end{equation}
where the containment follows, as $\zeta\in(0,1]$, from
$$
    \sqrt{x^2+y^2}\le\frac{\zeta}{32}\implies\left(|x|\le\frac{\zeta}{32}=\frac{1}{2}\cdot\frac{\zeta/4}{4}\quad\text{and}\quad|y|\le\frac{\zeta}{32}\le\frac{1}{32}\le\frac{1}{2}\cdot\frac{1}{4}\right).
$$

\begin{proof}
    If $(x,y)^{\T}\notin \bbB_k$ (i.e., $\sqrt{(x-x_k)^2+y^2}>\zeta/32$), 
    then we have $p_2(x-x_k,y;\zeta/32)=0$ by (\ref{eq:def-p2}), and hence $f_k(x,y;\zeta)=h(x-x_k,y;\zeta/4)$, as desired.
    If $(x,y)^{\T}\in[x_k-\zeta/16,\allowbreak x_k+\zeta/16]\allowbreak\times\allowbreak[-{1}/{4},\allowbreak {1}/{4}]$,
    then it follows from (\ref{eq:h-equals}) that $h(x-x_k,y;\zeta/4)=-2^{-1}(x-x_k)^2-2^{-1}y^2$, and hence
    $$
    \begin{aligned}
        f_k(x,y;\zeta):={}& p_2(x-x_k,y;\zeta/32)\cdot\left(-\frac{1}{2} (x-x_k)^2\right)\\
        {}&+\bigl(1-p_2(x-x_k,y;\zeta/32)\bigr)\cdot\left(-\frac{1}{2} (x-x_k)^2-\frac{1}{2} y^2\right)=g(x-x_k,y;\zeta/32),
    \end{aligned}
    $$
    as desired.   
    The proof is complete.
\end{proof}

With Lemma~\ref{lma:trivial-obs} in place, we are now 
ready to derive a piecewise representation of $f(\bullet,\bullet;\zeta)$,
revealing that it results from the local surgeries performed on the ambient function around every $\bx_k$.
\begin{proposition}\label{prop:covering}
    Consider the open sets $\bbH_0,\ldots,\bbH_m$ and $\bbP_0,\ldots,\bbP_m$ defined by
    $$
        \bbH_k:=\left(\left(x_{k-1}+\frac{\zeta}{4},x_{k+1}-\frac{\zeta}{4}\right)\times\R\right)\cap\bbB_k^{\complement}
        \quad\text{and}\quad
        \bbP_k:=\left(x_k-\frac{\zeta/4}{4},x_k+\frac{\zeta/4}{4}\right)\times\left(-\frac{1}{4},\frac{1}{4}\right).
    $$
    Their union constitutes an open cover of $\R^2$; i.e., $\bigcup_{k=0}^m\bbH_k\cup\bigcup_{k=0}^m\bbP_k=\R^2$. Besides, we have
    $$
        f(x,y;\zeta)=
        \begin{dcases}
            h(x-x_k,y;\zeta/4), & (x,y)^{\T}\in\bbH_k,\\
            g(x-x_k,y;\zeta/32), & (x,y)^{\T}\in\bbP_k.
        \end{dcases}
    $$
\end{proposition}

\begin{figure}[!ht]
    \centering
    \resizebox{\linewidth}{!}{
    \input{covering.tikz}
    }
    \caption{An illustration of $\bbB_k$, $\bbP_k$, and (implicitly) $\bbH_k$ 
    for $k=0,\ldots,m$
    when $m=2$
    and $\zeta=1$.
    }
    \label{fig:covering}
\end{figure}
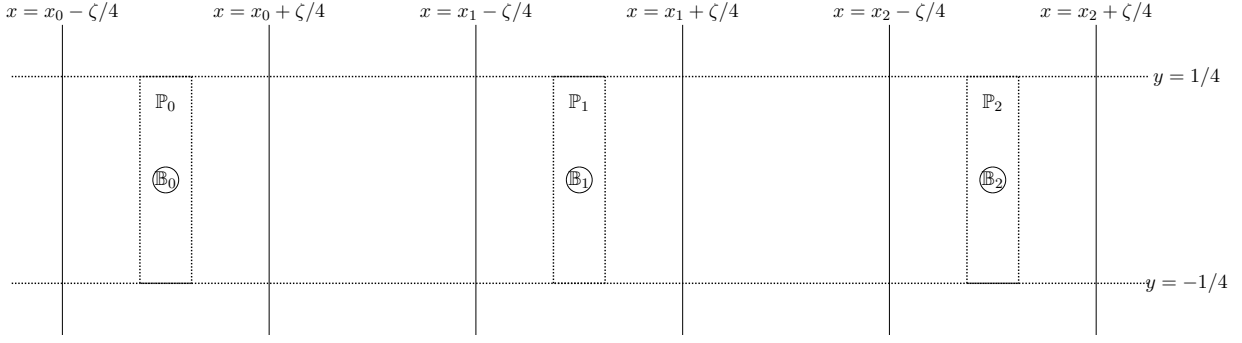

\begin{proof}
    We shall first prove $\bigcup_{k=0}^m\bbH_k\cup\bigcup_{k=0}^m\bbP_k=\R^2$; as $\R^2=\bigl(\bigcup_{k=0}^m \bbB_k\bigr)\cup\bigl(\bigcup_{k=0}^m \bbB_k\bigr)^{\complement}$, it suffices to cover these two sets by $\bigcup_{k=0}^m\bbP_k$ and $\bigcup_{k=0}^m\bbH_k$, respectively. By (\ref{eq:Bk}), we have $\bbB_k\subseteq\bbP_k$ for all $k=0,\ldots,m$, and hence $\bigcup_{k=0}^m \bbB_k\subseteq\bigcup_{k=0}^m\bbP_k$; 
    it remains to show $\bigl(\bigcup_{k=0}^m \bbB_k\bigr)^{\complement}\subseteq\bigcup_{k=0}^m\bbH_k$. To this end, let $(x,y)^{\T}\in\bigl(\bigcup_{k=0}^m \bbB_k\bigr)^{\complement}$ be arbitrary; i.e., we have $(x,y)^{\T}\notin\bbB_k$ for all $k=0,\ldots,m$. Suppose that $|x-x_{k^{\prime}}|\le \zeta/4$ for some $k^{\prime}=0,\ldots,m$. Then, 
    we must have $x\in(x_{k^{\prime}-1}+{\zeta}/{4},x_{k^{\prime}+1}-{\zeta}/{4})$, and hence $(x,y)^{\T}\in\bbH_{k^{\prime}}\subseteq\bigcup_{k=0}^m\bbH_k$, as desired. On the other hand, suppose that $|x-x_k|> \zeta/4$ for all $k=0,\ldots,m$. Then, there must exist some $k^{\prime}=0,\ldots,m$ such that 
    $$
        x\in(x_{k^{\prime}-1}+{\zeta}/{4},x_{k^{\prime}}-{\zeta}/{4})\cup(x_{k^{\prime}}+{\zeta}/{4},x_{k^{\prime}+1}-{\zeta}/{4})\subseteq(x_{k^{\prime}-1}+{\zeta}/{4},x_{k^{\prime}+1}-{\zeta}/{4});
    $$
    i.e., $(x,y)^{\T}\in\bbH_{k^{\prime}}\subseteq\bigcup_{k=0}^m\bbH_k$, as desired. This establishes the covering property.
    
    We next turn to study the piecewise behavior of $f(\bullet,\bullet;\zeta)$. Suppose that $(x,y)^{\T}\in\bbH_k$ for some $k=0,\ldots,m$. Then, we divide the subsequent discussions into two cases.
    \begin{itemize}
        \item If $|x-x_k|\le \zeta/4$, then we know by (\ref{eq:def-f}) that $f(x,y;\zeta)=f_k(x,y;\zeta)$. However, as $(x,y)^{\T}\notin\bbB_k$ by the definition of $\bbH_k$, it also follows from (\ref{eq:fact-fk}) that $f_k(x,y;\zeta)=h(x-x_k,y;\zeta/4)$. By putting the above two pieces together, we have $f(x,y;\zeta)=h(x-x_k,y;\zeta/4)$, as desired.

        \item If $|x-x_k|> \zeta/4$, then we know by the definition of $\bbH_k$ that 
        $$
            x\in(x_{k-1}+{\zeta}/{4},x_k-{\zeta}/{4})\cup(x_k+{\zeta}/{4},x_{k+1}-{\zeta}/{4});
        $$
        as a result, there does not exist any $k^{\prime}=0,\ldots,m$ such that $|x-x_{k^{\prime}}|\le \zeta/4$. This, together with (\ref{eq:def-f}), implies that $f(x,y;\zeta)=q(y;1)$. However, as $|x-x_k|> \zeta/4$, we also know by (\ref{eq:def-h}) and (\ref{eq:def-q}) that $h(x-x_k,y;\zeta/4)=q(x-x_k;\zeta/4)+q(y;1)=q(y;1)$. Hence, we have $f(x,y;\zeta)=h(x-x_k,y;\zeta/4)$, as desired.
    \end{itemize}
    On the other hand, suppose that $(x,y)^{\T}\in\bbP_k$ for some $k=0,\ldots,m$. Then, we have $|x-x_k|\le\zeta/16\le\zeta/4$, and hence $f(x,y;\zeta)=f_k(x,y;\zeta)$ by (\ref{eq:def-f}). However, by combining the definition of $\bbP_k$ and (\ref{eq:fact-fk}), we also know that $f_k(x,y;\zeta)=g(x-x_k,y;\zeta/32)$. Hence, we have $f(x,y;\zeta)=g(x-x_k,y;\zeta/32)$, as desired. The whole proof is now complete.
\end{proof}


As an immediate consequence of Proposition~\ref{prop:covering}, some desired properties of $f(\bullet,\bullet;\zeta)$ follow.

\begin{corollary}\label{cor:property-1}
    The following two statements hold.
    \begin{itemize}
        \item The function $f(\bullet,\bullet;\zeta)$ is $C^{1}$ with $\nabla f(\bullet,\bullet;\zeta)$ being $4.086$-Lipschitz continuous.
        
        \item The function $f(\bullet,\bullet;\zeta)$ is compatible with $\mathbfcal{O}$ along $\bx_0,\ldots,\bx_m$, and is $C^3$ around $\bx_0,\ldots,\bx_m$.

    \end{itemize}
\end{corollary}

\begin{proof}
Let $(x,y)^{\T}\in\R^2$ be arbitrary. We know from Proposition~\ref{prop:covering} that $(x,y)^{\T}$ must belong to some $\bbH_k$ or $\bbP_k$. If $(x,y)^{\T}\in\bbH_k$, then it follows from the facts that $\bbH_k$ is open and $h(\bullet-x_k,\bullet;\zeta/4)$ is $C^1$, Proposition~\ref{prop:covering}, and Lemma~\ref{lma:lip-h} that $f(\bullet,\bullet;\zeta)$ is $C^1$ around $(x,y)^{\T}$ with 
$$
    \operatorname{lip}\nabla f(x,y;\zeta)=\operatorname{lip}\nabla h(x-x_k,y;\zeta/4)\le 1\le 4.086;
$$
recall from~\citep[Definition~9.1(b)]{rockafellar2009variational} that $\operatorname{lip}\nabla f(\bullet,\bullet;\zeta)$ is defined only locally. If $(x,y)^{\T}\in\bbP_k$, then it follows from the facts that $\bbP_k$ is open and $g(\bullet-x_k,\bullet;\zeta/32)$ is $C^3$, Proposition~\ref{prop:covering}, and Corollary~\ref{cor:spectral-bound2} that $f(\bullet,\bullet;\zeta)$ is $C^1$ around $(x,y)^{\T}$ with 
$$
    \operatorname{lip}\nabla f(x,y;\zeta)=\operatorname{lip}\nabla g(x-x_k,y;\zeta/32)\le 4.086.
$$ 
In summary,
it follows that $f(\bullet,\bullet;\zeta)$ is $C^{1}$ with $\operatorname{lip}\nabla f(x,y;\zeta)\le 4.086$ for all $x,y\in\R$; i.e., $\nabla f(\bullet,\bullet;\zeta)$ is Lipschitz continuous with constant $4.086$~\citep[Theorem~9.2]{rockafellar2009variational}, as desired.

To see the other claim, it suffices to combine (\ref{eq:def-f}) and (\ref{eq:fact-fk}) with Lemma~\ref{lma:oracle} and the $C^3$-smoothness of $g(\bullet,\bullet;\delta)$. The proof is complete.
\end{proof}

Another consequence of Proposition~\ref{prop:covering} is a strong stationarity refutation property of $f(\bullet,\bullet;\zeta)$.

\begin{corollary}[Band refutation]\label{cor:band-refutation}
Let
$\bw:=(10,11)^{\T}/\sqrt{221}\in\bbS^{2}$. We have 
    $$
        \max\mleft\{\langle\bz,\bw\rangle:\bz\in\operatorname{conv}\mleft(\bigcup_{(x,y)^{\T}\in\R\times(-1/4,1/4)}\partial_{\dC}^2 f(x,y;\zeta)(\bw)\mright)\mright\}\le -0.0092714.
    $$
\end{corollary}

\begin{proof}
    Let $\bz\in\operatorname{conv}\bigl(\bigcup_{(x,y)^{\T}\in\R\times(-1/4,1/4)}\partial_{\dC}^2 f(x,y;\zeta)(\bw)\bigr)$ be arbitrary. By the definition of convex hulls, we have $\bz=\sum_{k^{\prime}=1}^{m^{\prime}} t_{k^{\prime}}\cdot\bz_{k^{\prime}}$ for some $m^{\prime}\ge 1$, $\bz_{k^{\prime}}\in\bigcup_{(x,y)^{\T}\in\R\times(-1/4,1/4)}\partial_{\dC}^2 f(x,y;\zeta)(\bw)$, and $t_{k^{\prime}}\ge 0$ with $\sum_{k^{\prime}=1}^{m^{\prime}} t_{k^{\prime}}=1$. Besides, we also know by the definition of $\bz_{k^{\prime}}$ that there further exist some $x_{k^{\prime}}^{\prime}\in\R$ and $y_{k^{\prime}}^{\prime}\in(-1/4,1/4)$ such that $\bz_{k^{\prime}}\in\partial_{\dC}^2 f(x_{k^{\prime}}^{\prime},y_{k^{\prime}}^{\prime};\zeta)(\bw)$. Hence, it follows that
    $$
    \begin{aligned}
        \langle\bz,\bw\rangle=\sum_{k^{\prime}=1}^{m^{\prime}} t_{k^{\prime}}\cdot\langle\bz_{k^{\prime}},\bw\rangle\le\sum_{k^{\prime}=1}^{m^{\prime}} t_{k^{\prime}}\cdot\max\mleft\{\langle\bz,\bw\rangle:\bz\in\partial_{\dC}^2 f(x_{k^{\prime}}^{\prime},y_{k^{\prime}}^{\prime};\zeta)(\bw)\mright\};
    \end{aligned}
    $$
    this motivates us to study $\max\bigl\{\langle\bz,\bw\rangle:\bz\in\partial_{\dC}^2 f(x_{k^{\prime}}^{\prime},y_{k^{\prime}}^{\prime};\zeta)(\bw)\bigr\}$
    instead. However, as $x_{k^{\prime}}^{\prime}\in\R$ and $y_{k^{\prime}}^{\prime}\in(-1/4,1/4)$, it further suffices to study $\max\bigl\{\langle\bz,\bw\rangle:\bz\in\partial_{\dC}^2 f(x,y;\zeta)(\bw)\bigr\}$ for some arbitrary $x\in\R$ and $y\in(-1/4,1/4)$.
    
    Let $(x,y)^{\T}\in\R\times(-1/4,1/4)$ be arbitrary. We know from Proposition~\ref{prop:covering} that $(x,y)^{\T}$ must belong to some $\bbH_k$ or $\bbP_k$. If $(x,y)^{\T}\in\bbH_k$, then it follows from Propositions~\ref{prop:covering} and~\ref{prop:ambient-refutation} 
    that
    $$
        \max\mleft\{\langle\bz,\bw\rangle:\bz\in\partial_{\dC}^2 f(x,y;\zeta)(\bw)\mright\}=\max\mleft\{\langle\bz,\bw\rangle:\bz\in\partial_{\dC}^2 h(x-x_k,y;\zeta/4)(\bw)\mright\}\le -0.095.
    $$
    If $(x,y)^{\T}\in\bbP_k$, then it follows from Proposition~\ref{prop:covering} and Corollary~\ref{cor:uniformity} that
    $$
        \max\mleft\{\langle\bz,\bw\rangle:\bz\in\partial_{\dC}^2 f(x,y;\zeta)(\bw)\mright\}=\langle\nabla^2 g(x-x_k,y;\zeta/32),\bw\bw^{\T}\rangle\le-0.0092714.
    $$
    By combining the above two pieces, we have $\max\bigl\{\langle\bz,\bw\rangle:\bz\in\partial_{\dC}^2 f(x,y;\zeta)(\bw)\bigr\}\le-0.0092714$ for all $(x,y)^{\T}\in\R\times(-1/4,1/4)$, which implies that $\max\bigl\{\langle\bz,\bw\rangle:\bz\in\partial_{\dC}^2 f(x_{k^{\prime}}^{\prime},y_{k^{\prime}}^{\prime};\zeta)(\bw)\bigr\}\le-0.0092714$ for all $k^{\prime}=1,\ldots,m^{\prime}$, and thus $\langle\bz,\bw\rangle\le -0.0092714$, as desired. We are done.
\end{proof}

With the previous two results in place and noting that for all $\bw\in\R^2$ and $k=0,\ldots,m$,
$$
    \operatorname{conv}\mleft(\bigcup_{(x,y)^{\T}\in(x_k,0)^{\T}+\bbB^2/8}\partial_{\dC}^2 f(x,y;\zeta)(\bw)\mright)\subseteq\operatorname{conv}\mleft(\bigcup_{(x,y)^{\T}\in\R\times(-1/4,1/4)}\partial_{\dC}^2 f(x,y;\zeta)(\bw)\mright),
$$
all that remains is to bound the initial optimality gap $f(0,0;\zeta)-\min\{f(x,y;\zeta):x,y\in\R\}$. 
\begin{lemma}
    We have $f(0,0;\zeta)-\min\{f(x,y;\zeta):x,y\in\R\}\le 17/256$.
\end{lemma}

We remark that as $f(0,0;\zeta)=0$ (as $(0,0)^{\T}$ is some iterate and $f(\bullet,\bullet;\zeta)$ is compatible with $\mathbfcal{O}$, whose zeroth-order output is always $0$, along all the iterates; see Corollary~\ref{cor:property-1}), it suffices to verify that $f(x,y;\zeta)\ge -{17}/{256}$ for all $x,y\in\R$.

\begin{proof}
    Since $f(\bullet,\bullet;\zeta)$ takes values only from $f_k(\bullet,\bullet;\zeta)$ and $q(\bullet;1)$, and $q(x;\gamma)\ge-\gamma^2/16$ for all $x\in\R$, it suffices to bound $f_k(\bullet,\bullet;\zeta)$. Let us divide the subsequent analysis into two cases.
    \begin{itemize}
        \item If $(x,y)^{\T}\notin (x_k,0)^{\T}+{\zeta\bbB^2}/{32}$, then it immediately follows from (\ref{eq:fact-fk}) and $\zeta\le 1$ that 
        $$
            f_k(x,y;\zeta)=h(x-x_k,y;\zeta/4)\ge -{1}/{256}-{1}/{16}=-{17}/{256}.
        $$

        \item If $(x,y)^{\T}\in (x_k,0)^{\T}+\zeta\bbB^2/32$, then it follows from (\ref{eq:fact-fk}) and $\zeta\le 1$ again that
            $$
            \begin{aligned}
                f_k(x,y;\zeta)={}&g(x-x_k,y;\zeta/32)\\
                ={}&p_2(x-x_k,y;\zeta/32)\cdot\left(-\frac{1}{2} (x-x_k)^2\right)\\
                {}&+\bigl(1-p_2(x-x_k,y;\zeta/32)\bigr)\cdot\left(-\frac{1}{2} (x-x_k)^2-\frac{1}{2} y^2\right)\\
                \ge{}& {-\frac{1}{2}}\cdot(\zeta/32)^2-\frac{1}{2}\cdot(\zeta/32)^2 \ge -\frac{1}{1024}\ge -\frac{17}{256}.
            \end{aligned}
            $$
    \end{itemize}
    In summary, we see that $f_k(x,y;\zeta)\ge -{17}/{256}$ for all $x,y\in\R$, and thus $f(x,y;\zeta)\ge -{17}/{256}$ for all $x,y\in\R$, as desired. The proof is complete.
\end{proof}

\clearpage

\section{Conjecture on extension to local oracles}
At first glance, it appears rather straightforward to extend Theorem~\ref{thm:lb-1} and Corollary~\ref{cor:lb-1} from second-order oracles to local oracles (see, e.g.,~\citep[Section~2.3]{tian2024no}): It should suffice to slightly flatten the landscape of $p_1(\bullet;\delta)$ near $0$; e.g., one may consider the 
variant $\overline{p}_1(\bullet;\delta,\lambda):\R\rightarrow\R$ of $p_1(\bullet;\delta)$ 
defined by
$$
    \overline{p}_1(x;\delta,\lambda):=
    \begin{dcases}
        1, & |x|\le\lambda\delta,\\
        \left(\frac{(|x|-\lambda\delta)^2}{(\delta-\lambda\delta)^2}-1\right)^4, & \lambda\delta\le|x|\le\delta,\\
        0, & \textnormal{otherwise},
    \end{dcases}
$$
where 
$\lambda\in(0,1)$
determines the extent of such a flattening. 

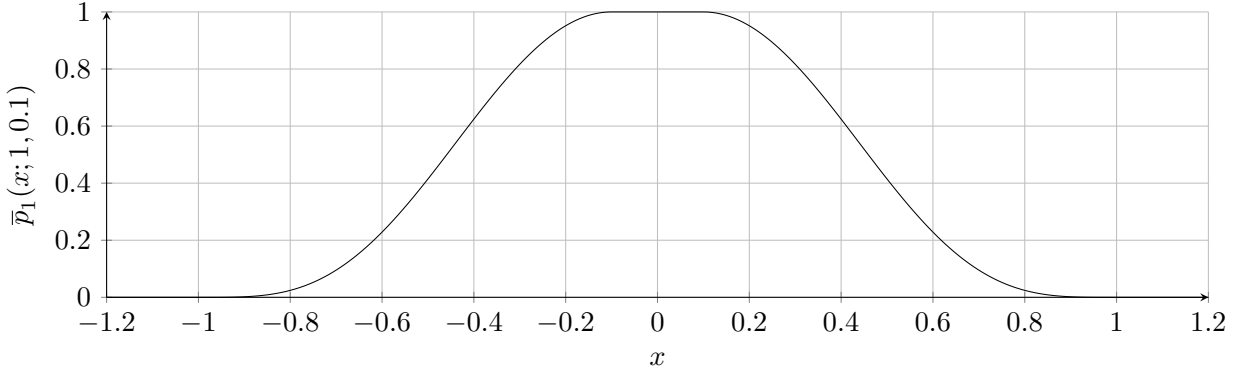
\begin{figure}[!ht]
\resizebox{\linewidth}{!}{%
\centering
\pgfmathsetmacro{\mylam}{0.1}
\begin{tikzpicture}[
  declare function={
    p1bar(\x) = (abs(\x) < \mylam) * 1
              + (abs(\x) >= \mylam && abs(\x) < 1) * ((abs(\x)-\mylam)^2/(1-\mylam)^2-1)^4
              + (abs(\x) >= 1) * 0;
    m(\x) = ((abs(\x)-\mylam)^2/(1-\mylam)^2-1)^4;
  }
]
\begin{axis}[
    axis lines = left,
    grid=major,
    xlabel = $x$,
    ylabel = {$\overline{p}_1(x;1,0.1)$},
height=0.33*\linewidth,
width=\linewidth,
]
\addplot [
    domain=-1.2:1.2, 
    samples=500,
    color=black,
]
{p1bar(x)};
\end{axis}
\end{tikzpicture}}
\caption{The landscape of $\overline{p}_1(\bullet;1,0.1)$.}
\label{fig:p1bar-func}
\end{figure}


Conceivably, as soon as $\lambda$ is sufficiently small, $\overline{p}_1(\bullet;\delta,\lambda)$ should almost coincide with $p_1(\bullet;\delta)$, thereby inheriting all its desired properties;
in fact, numerical evidence suggests that setting $\lambda=0.1$ already suffices by far.
However, as it turns out, the analysis of the $\overline{p}_1(\bullet;\delta,\lambda)$-based construction is entirely a nightmare to carry out.
This is mainly due to the fact that 
the corresponding $\overline{p}_2(\bullet,\bullet;\delta,\lambda)$ defined by $\overline{p}_2(x,y;\delta,\lambda):=\overline{p}_1\bigl(\sqrt{x^2+y^2};\delta,\lambda\bigr)$ is no longer a piecewise polynomial as $p_2(\bullet,\bullet;\delta)$, but rather a piecewise polynomial in \textit{norms}. 
That said, we remain firmly convinced that the extension is valid.

\begin{conjecture}\label{conj:local-oracles}
    Theorem~\ref{thm:lb-1} and Corollary~\ref{cor:lb-1} remain valid
    (in substance) even under
    local oracles.
\end{conjecture}


\end{document}

%% file: sample.tikzstyles
\tikzstyle{red dot}=[fill=red, draw=black, shape=circle]
\tikzstyle{green dot}=[fill=green, draw=black, shape=circle]
\tikzstyle{medium box}=[fill=white, draw=black, shape=rectangle, minimum width=1cm, minimum height=1cm]
\tikzstyle{0.1pt}=[fill=black, draw=black, shape=circle, minimum size=0.1pt, inner sep=0pt]
\tikzstyle{0.5pt}=[fill=black, draw=black, shape=circle, minimum size=0.5pt, inner sep=0pt]
\tikzstyle{1pt}=[fill=black, draw=black, shape=circle, minimum size=1pt, inner sep=0pt]
\tikzstyle{2pt}=[fill=black, draw=black, shape=circle, minimum size=2pt, inner sep=0pt]
\tikzstyle{3pt}=[fill=black, draw=black, shape=circle, minimum size=3pt, inner sep=0pt]
\tikzstyle{3pthide}=[fill=black, draw=black, shape=circle, minimum size=3pt, inner sep=0pt, opacity=0.1]
\tikzstyle{5pt}=[fill=black, draw=black, shape=circle, minimum size=5pt, inner sep=0pt]
\tikzstyle{big circle}=[fill=none, draw=black, shape=circle, minimum size=20cm, inner sep=0pt, ultra thick]
\tikzstyle{0.5none}=[fill=none, draw=none, shape=circle, scale=0.5]
\tikzstyle{small circle}=[fill=none, draw=black, shape=circle, minimum size=5cm, inner sep=0pt, ultra thick]

\tikzstyle{directional}=[>=stealth, ->]
\tikzstyle{thick direc}=[>=stealth, ->, very thick]
\tikzstyle{dashes}=[-, densely dotted, thick]
\tikzstyle{wavy}=[-, snake it, thick]
\tikzstyle{big dashes}=[-, thick, dashed, dash pattern=on 4mm off 2mm, fill=cyan]
\tikzstyle{blue directional}=[draw=blue, ->, very thick, >=stealth]
\tikzstyle{red directional}=[draw=red, ->, very thick, >=stealth]
\tikzstyle{green directional}=[draw=green, ->, very thick, >=stealth]
\tikzstyle{thin line}=[-, very thin]
\tikzstyle{purple directional}=[->, draw=magenta, very thick, >=stealth]
\tikzstyle{pure shade}=[-, draw=none, fill={rgb,255: red,191; green,191; blue,191}, opacity=0.5]
\tikzstyle{pure rshade}=[-, draw=none, fill={rgb,255: red,255; green,0; blue,0}, opacity=0.5]
\tikzstyle{pure gshade}=[-, draw=none, fill={rgb,255: red,0; green,255; blue,0}, opacity=0.5]
\tikzstyle{blue line}=[-, fill=none, draw=blue, very thick]
\tikzstyle{purple line}=[-, draw=magenta, very thick]
\tikzstyle{line shade}=[-, draw=black, fill={rgb,255: red,191; green,191; blue,191}]
\tikzstyle{line shade 2}=[-, draw=black, fill={rgb,255: red,128; green,128; blue,128}]
\tikzstyle{thick line}=[-, very thick]
\tikzstyle{tlhide}=[-, very thick, opacity=0.1]
\tikzstyle{green shadedline}=[-, draw=green, fill=green, opacity=0.3]
\tikzstyle{blue shadedline}=[-, draw=blue, fill=blue, opacity=0.3]
\tikzstyle{red shadedline}=[-, draw=red, fill=red, opacity=0.5]
\tikzstyle{cyan shadedline}=[-, draw=cyan, fill=cyan, opacity=0.3]
\tikzstyle{magenta shadedline}=[-, draw=magenta, fill=magenta, opacity=0.2]
\tikzstyle{yellow shadedline}=[-, draw=yellow, fill=yellow, opacity=0.3]
\tikzstyle{violet shadedline}=[-, draw=violet, fill=violet, opacity=0.3]
\tikzstyle{teal shadedline}=[-, draw=teal, fill=teal, opacity=0.3]
\tikzstyle{olive shadedline}=[-, draw=olive, fill=olive, opacity=0.2]

%% file: covering.tikz
\begin{tikzpicture}
	\begin{pgfonlayer}{nodelayer}
		\node [style=none] (3) at (-16.5, 0) {};
		\node [style=none] (4) at (-16, 0.5) {};
		\node [style=none] (5) at (-15.5, 0) {};
		\node [style=none] (6) at (-16, -0.5) {};
		\node [style=none] (8) at (15.5, 0) {};
		\node [style=none] (9) at (16, 0.5) {};
		\node [style=none] (10) at (16.5, 0) {};
		\node [style=none] (11) at (16, -0.5) {};
		\node [style=none] (13) at (-0.5, 0) {};
		\node [style=none] (14) at (0, 0.5) {};
		\node [style=none] (15) at (0.5, 0) {};
		\node [style=none] (16) at (0, -0.5) {};
		\node [style=none] (22) at (-1, 4) {};
		\node [style=none] (23) at (1, 4) {};
		\node [style=none] (24) at (1, -4) {};
		\node [style=none] (25) at (-1, -4) {};
		\node [style=none] (26) at (15, 4) {};
		\node [style=none] (27) at (17, 4) {};
		\node [style=none] (28) at (17, -4) {};
		\node [style=none] (29) at (15, -4) {};
		\node [style=none] (30) at (-17, 4) {};
		\node [style=none] (31) at (-15, 4) {};
		\node [style=none] (32) at (-15, -4) {};
		\node [style=none] (33) at (-17, -4) {};
		\node [style=none] (34) at (-22, 4) {};
		\node [style=none] (35) at (-22, -4) {};
		\node [style=none] (36) at (22, 4) {};
		\node [style=none] (37) at (22, -4) {};
		\node [style=none] (38) at (23.5, 4) {$y=1/4$};
		\node [style=none] (39) at (23.5, -4) {$y=-1/4$};
		\node [style=none] (42) at (-20, 6) {};
		\node [style=none] (43) at (-12, 6) {};
		\node [style=none] (44) at (-12, -6) {};
		\node [style=none] (45) at (-20, -6) {};
		\node [style=none] (46) at (12, 6) {};
		\node [style=none] (47) at (20, 6) {};
		\node [style=none] (48) at (20, -6) {};
		\node [style=none] (49) at (12, -6) {};
		\node [style=none] (50) at (-4, 6) {};
		\node [style=none] (51) at (4, 6) {};
		\node [style=none] (52) at (4, -6) {};
		\node [style=none] (53) at (-4, -6) {};
		\node [style=none] (54) at (-12, 6.5) {$x=x_0+\zeta/4$};
		\node [style=none] (55) at (-20, 6.5) {$x=x_0-\zeta/4$};
		\node [style=none] (56) at (4, 6.5) {$x=x_1+\zeta/4$};
		\node [style=none] (57) at (-4, 6.5) {$x=x_1-\zeta/4$};
		\node [style=none] (58) at (20, 6.5) {$x=x_2+\zeta/4$};
		\node [style=none] (59) at (12, 6.5) {$x=x_2-\zeta/4$};
		\node [style=none] (60) at (-16, 0) {$\mathbb{B}_0$};
		\node [style=none] (61) at (0, 0) {$\mathbb{B}_1$};
		\node [style=none] (62) at (16, 0) {$\mathbb{B}_2$};
		\node [style=none] (63) at (16, 3) {$\mathbb{P}_2$};
		\node [style=none] (64) at (-16, 3) {$\mathbb{P}_0$};
		\node [style=none] (65) at (0, 3) {$\mathbb{P}_1$};
	\end{pgfonlayer}
	\begin{pgfonlayer}{edgelayer}
		\draw [bend right=45] (6.center) to (5.center);
		\draw [bend right=45] (5.center) to (4.center);
		\draw [bend right=45] (4.center) to (3.center);
		\draw [bend right=45] (3.center) to (6.center);
		\draw [bend right=45] (11.center) to (10.center);
		\draw [bend right=45] (10.center) to (9.center);
		\draw [bend right=45] (9.center) to (8.center);
		\draw [bend right=45] (8.center) to (11.center);
		\draw [bend right=45] (16.center) to (15.center);
		\draw [bend right=45] (15.center) to (14.center);
		\draw [bend right=45] (14.center) to (13.center);
		\draw [bend right=45] (13.center) to (16.center);
		\draw [style=dashes] (22.center) to (23.center);
		\draw [style=dashes] (23.center) to (24.center);
		\draw [style=dashes] (24.center) to (25.center);
		\draw [style=dashes] (25.center) to (22.center);
		\draw [style=dashes] (26.center) to (27.center);
		\draw [style=dashes] (27.center) to (28.center);
		\draw [style=dashes] (28.center) to (29.center);
		\draw [style=dashes] (29.center) to (26.center);
		\draw [style=dashes] (30.center) to (31.center);
		\draw [style=dashes] (31.center) to (32.center);
		\draw [style=dashes] (32.center) to (33.center);
		\draw [style=dashes] (33.center) to (30.center);
		\draw [style=dashes] (36.center) to (34.center);
		\draw [style=dashes] (37.center) to (35.center);
		\draw (45.center) to (42.center);
		\draw (43.center) to (44.center);
		\draw (49.center) to (46.center);
		\draw (47.center) to (48.center);
		\draw (53.center) to (50.center);
		\draw (51.center) to (52.center);
	\end{pgfonlayer}
\end{tikzpicture}